\documentclass{article}
\usepackage[a4paper, left = 1.5cm, right = 2cm, top = 1.5cm, bottom = 2cm]{geometry}

\pdfoutput=1
\usepackage{amssymb}
\usepackage{amsmath}
\usepackage{adjustbox}
\usepackage{multirow}
\usepackage{enumerate}
\usepackage[colorlinks,
            linkcolor=blue,
            anchorcolor=blue,
            citecolor=blue
            ]{hyperref}
\usepackage{amsthm}
\usepackage{authblk} 
\usepackage{booktabs}
\usepackage{caption}
\makeatletter
\newcommand{\rmnum}[1]{\romannumeral #1}
\newcommand{\Rmnum}[1]{\expandafter\@slowromancap\romannumeral #1@}
\makeatother

\newtheorem{theorem}{Theorem}
\newtheorem{lemma}{Lemma}
\newtheorem{proposition}{Proposition}

\newtheorem{definition}{Definition}
\newtheorem{assumption}{Assumption}
\newtheorem{example}{Example}

\newtheorem{remark}{Remark}

\begin{document}

\title{Higher-order exponential Runge-Kutta Galerkin finite element method for semilinear parabolic problems with nonsmooth data}

\author[1]{Shuo Yang\thanks{Email: 2112414027@mail2.gdut.edu.cn}}
\author[1]{Runjie Zhang\thanks{Email: 2112314030@mail2.gdut.edu.cn}}
\author[2]{Zhe Yu\thanks{Email: yuzhe@hit.edu.cn}}
\author[1]{Jinwei Fang\thanks{Corresponding author. Email: fangjinwei@gdut.edu.cn}}

\affil[1]{School of Mathematics and Statistics, Guangdong University of Technology, Guangzhou, Guangdong 510006, China}
\affil[2]{School of Science, Harbin Institute of Technology (Shenzhen), Shenzhen 518055, China}

\maketitle

\abstract{We develop a rigorous numerical analysis framework for a class of semilinear parabolic problems with nonsmooth initial data. We employ a linear Galerkin finite element method for spatial discretization coupled with a high-order explicit exponential Runge-Kutta (EERK) temporal integration scheme. In contrast to conventional smooth error analysis, the nonsmooth case lacks a priori estimates for the higher-order derivatives of both the nonlinear term and the exact solution. By combining analytic semigroup techniques with fractional power space theory, we establish rigorous bounds for these derivatives. Finally, our analysis proves that the $p$th-order EERK method achieves a convergence rate of $\min(1 + \gamma/2 + \rho_1(\gamma)/2,\:p)$, where $\gamma$ characterizes the initial data regularity and $\rho_1(\gamma)$ quantifies the boundedness of the nonlinearity's first Fréchet derivative. Numerical experiments confirm the sharpness of these estimates.}

\newcommand{\keywords}[1]{\textbf{Keywords:} #1}

\newcommand{\pacs}[1]{\textbf{Mathematics Subject Classification (2010):} #1}

\section{Introduction}
We consider the following semilinear parabolic equation
\begin{equation}\label{eq1}
	\left\{
	\begin{aligned}
		&\frac{du(t)}{dt} = Au(t) + f(u(t)), \qquad 0<t\leq T, \\
		&u(0) = u_0,
	\end{aligned}\right.
\end{equation}
in the Hilbert space $X=L^2(\Omega)$, where $\Omega$ is a convex polygonal domain in $\mathbb{R}^d$ for $d \in \{1,2,3\}$. The linear operator $A:D(A)\subset X \rightarrow X$ is self-adjoint and negative definite, and it generates an analytic semigroup $S(t)=e^{At}$ for $t\geq 0$. The nonlinear term $f:D(A^{\eta/2})\rightarrow X$ is a Nemytskii operator induced by a continuous function on $\mathbb{R}$, where $0 <\eta< 2$. The initial value satisfies $u_0\in D(A^{\gamma/2})$ for $0<\gamma<2$. Further details on problem \eqref{eq1} are given in Section \ref{sec:AnaFrame}.

This work aims to analyze the fully discrete error of this problem. Linear finite elements and high-order exponential Runge-Kutta (EERK) methods are employed for spatial and temporal discretizations, respectively, with an emphasis on the latter. Exponential integrators have proven to be highly effective for the time integration of parabolic equations. These integrators effectively mitigate the impact of stiffness by exactly handling the linear stiff terms within the scheme. With advancements in computational efficiency, significant efforts have been dedicated to the development of exponential integrators for semilinear parabolic problems. Various types of exponential integrators have been explored, including exponential Runge-Kutta methods \cite{ERK1,ERK2,luan14stiff}, exponential multistep methods \cite{EMult1,EMult2}, and exponential Rosenbrock methods \cite{ERB1,ERB2}. A comprehensive overview of exponential integrators is available in \cite{EI_hochbruck,EI_review}.

Conventional error estimation techniques for numerical methods require boundedness assumptions on the derivatives of the solution and the nonlinear term in semilinear parabolic equations. These assumptions hold when the initial data and nonlinear terms are sufficiently smooth, and appropriate compatibility conditions are satisfied. However, if the initial data lacks smoothness, these boundedness conditions fail, leading to severe order reduction. To clearly illustrate this phenomenon, Table \ref{tab:example} presents the temporal discretization errors and convergence orders for the Field-Noyes model, which serves as a representative semilinear parabolic system, under initial conditions of varying regularities (see \cite{ERK_FNM}). Here, the parameter $\gamma$ characterizes the exact regularity of the initial data, satisfying $u_0 \in D(A^{\gamma/2})$.

\begin{table}[h]
	\centering
	\caption{Temporal discretization errors and convergence orders for the Field-Noyes model with initial data of different regularities $\gamma$ (where $u_0 \in D(A^{\gamma/2})$). Top: Second-order EERK method whose theory was established in \cite{ERK_FNM}; Bottom: Third-order EERK method, which is the main focus of this study.}
    \label{tab:example}
    \begin{adjustbox}{max width=\textwidth}
	\begin{tabular}{ccccccccccccc}
		\hline
		\multirow{2}{*}{$N$} & & \multicolumn{2}{c}{$\gamma=0.5-\varepsilon$} &&  \multicolumn{2}{c}{$\gamma=0.75-\varepsilon$}  && \multicolumn{2}{c}{$\gamma=1-\varepsilon$} && \multicolumn{2}{c}{$\gamma=1.5-\varepsilon$}
		\\ \cline{3-4} \cline{6-7} \cline{9-10} \cline{12-13} 
		~& & Error & Order &&  Error & Order &&  Error & Order &&  Error & Order 
		\\ \hline
		$2^5$ && 3.588E-04 & -- & ~ & 3.015E-05 & -- & ~ & 3.423E-05 & -- & ~ & 5.886E-05 & --  \\ 
        $2^6$ && 1.225E-04 & 1.551  & ~ & 8.525E-06 & 1.822 & ~ & 8.858E-06 & 1.950 & ~ & 1.513E-05 & 1.960 \\ 
        $2^7$ && 4.098E-05 & 1.579  & ~ & 2.436E-06 & 1.807 & ~ & 2.280E-06 & 1.958 & ~ & 3.855E-06 & 1.973 \\ 
        $2^8$ && 1.331E-05 & 1.622  & ~ & 6.988E-07 & 1.801 & ~ & 5.842E-07 & 1.965 & ~ & 9.742E-07 & 1.984 \\  
        \hline
        $2^4$ & ~ & 2.825E-04 & -- & ~ & 1.503E-05 & -- & ~ & 5.101E-06 & -- & ~ & 1.210E-05 & -- \\ 
        $2^5$ & ~ & 9.202E-05 & 1.618 & ~ & 4.096E-06 & 1.876 & ~ & 1.132E-06 & 2.172 & ~ & 2.203E-06 & 2.458  \\ 
        $2^6$ & ~ & 2.957E-05 & 1.638 & ~ & 1.140E-06 & 1.845 & ~ & 2.618E-07 & 2.112 & ~ & 3.976E-07 & 2.470 \\ 
        $2^7$ & ~ & 9.152E-06 & 1.692 & ~ & 3.200E-07 & 1.833 & ~ & 6.110E-08 & 2.099 & ~ & 6.928E-08 & 2.521  \\ 
        \hline
	\end{tabular}
    \end{adjustbox}
\end{table}

The upper part of Table \ref{tab:example} shows the numerical results using a second-order EERK method. The theoretical error analysis for this second-order scheme has already been rigorously established in our prior work \cite{ERK_FNM}, which proves that the convergence order strictly adapts to the initial regularity as $\min(1+\gamma, 2)$. Conversely, the lower part of the table displays the numerical results when applying a third-order EERK method. Compared to the second-order case, the theoretical analysis for this higher-order scheme is significantly more complex, as it requires delicately bounding the higher-order Fréchet derivatives of the nonlinear terms under nonsmooth data. Developing a comprehensive analytical framework to resolve this complexity and establishing sharp error estimates for general high-order EERK methods is the primary objective of the present study.

In recent decades, extensive research has been conducted on the nonsmooth error analysis of \emph{abstract} semilinear parabolic equations, covering both spatial and temporal discretization approaches. In the context of spatial discretization, Galerkin approximation techniques have been thoroughly analyzed, with comprehensive error estimates consolidated in the monograph \cite{GaleFEM}. For temporal discretization, various numerical schemes—including fully implicit, semi-implicit, exponential Rosenbrock, and implicit-explicit methods \cite{IRK1,IRK3,IRK2,mukam,24IMEX}—have been investigated, revealing the phenomenon of order reduction under nonsmooth initial conditions. However, these abstract analyses often yield suboptimal order estimates, primarily due to restrictive assumptions imposed on the nonlinear term. This gap motivates the necessity of rigorous nonsmooth error analysis tailored to \emph{specific} PDEs. 

For instance, the Navier-Stokes equations have been extensively studied, as summarized in \cite{21NSH1}. Notable contributions include first-order convergence results for the Euler implicit/explicit scheme in \cite{08He}, and suboptimal $1.5$-order convergence for a Crank-Nicolson/Adams-Bashforth scheme in \cite{10He}, both requiring $H^1$ initial regularity. Under weaker $L^2$ initial conditions, \cite{22Li_L2} proved first-order convergence for a variable-stepsize semi-implicit method. Additionally, \cite{18_Burges} extended the analysis to the Burgers equation, demonstrating $1.5$-order convergence under $H^1$ initial conditions. These error estimates are derived via energy method techniques.

In our prior work \cite{ERK_FNM}, by applying alternative methodologies, we obtained sharp error estimates for a reaction-diffusion equation with initial values in $D(A^{\gamma/2})$ ($0<\gamma<2$). As shown in the upper part of Table \ref{tab:example}, when $\gamma > 1$, the convergence order reaches the upper bound of $2$. A natural question arises: whether employing a third-order EERK method can further improve the convergence order and whether it exhibits behavior similar to the second-order method when $0<\gamma<1$. Experiments using a third-order method (detailed in \eqref{EERK3}) affirmatively answer both points (see the bottom of Table \ref{tab:example}). Moreover, the analytical framework developed in \cite{ERK_FNM} is sufficiently general to be extended to a broader class of nonlinear equations by representing the nonlinear terms via Nemytskii operators. 

The present work develops a comprehensive framework for the error analysis of the general EERK method \eqref{NumSche1}, applied to a class of semilinear parabolic problems with nonlinearities, such as polynomial and rational functions, satisfying Assumptions \ref{assum:nf_smooth}, \ref{assum:bound_compos} and \ref{ass:Frech_dominant}. The primary challenges, stemming from nonlinearities, are outlined in Remarks \ref{challenge_1} and \ref{challenge_2}. The analysis ultimately establishes a temporal convergence order of $\min(1+\gamma/2+\rho_1(\gamma)/2,\: p)$, where $\rho_1(\gamma)$ characterizes the boundedness of $f'(u(t))$. 

The paper is organized as follows. Section 2 establishes the abstract framework for the class of semilinear parabolic problems and analyzes their well-posedness. Building on this foundation, Section 3 develops the spatial discretization scheme and establishes spatial error estimates. Section 4 formulates the EERK scheme and proves its stability, followed by a detailed Taylor expansion analysis to examine the structure of higher-order local error terms. Section 5 derives the key estimates concerning the Fréchet derivatives of the Nemytskii operator, ultimately leading to the sharp convergence rates. Section 6 presents numerical experiments that validate the theoretical findings. Finally, Section 7 provides concluding remarks and discusses future research directions.

\section{Analysis framework}\label{sec:AnaFrame}
Let us begin by introducing standard notation. For $s\geq 0$, we denote by $\|\cdot\|_s$ the norm of the Sobolev spaces $H^s=H^s(\Omega)$ over the domain $\Omega$ (see, e.g., \cite{ABsPEE, multiplication21}). When $s=0$, $H^0$ is equivalent to $L^2=L^2(\Omega)$ with norm $\|\cdot\|$ and inner product $(\cdot,\cdot)$. The $L^\infty$ space consists of all bounded measurable functions on $\Omega$. For $s\geq 1$, the space $H_N^s=H_N^s(\Omega)$ denotes the Sobolev space $H^s$ subject to homogeneous Neumann boundary conditions. Given any real-valued function space $Y$ and an interval $I\subset \mathbb{R}$, we introduce the restricted space $Y|_I=\{v\in Y \mid v(x)\in I \text{ a.e. on }\Omega \}$. Throughout this paper, we denote by $C$ a generic positive constant and by $\varepsilon$ a sufficiently small positive number, both of which may vary across instances. 

With these notations established, we consider the semilinear parabolic problem \eqref{eq1}. The underlying space $X$ of \eqref{eq1} is $L^2$. Let $\mathcal{L}(X)$ denote the Banach space of bounded linear operators from $X$ to $X$ with operator norm $\Vert\cdot\Vert_{\mathcal{L}(X)}$. The linear operator $-A$ is generated by the sesquilinear form
$$
a(u,v)=\sum_{i,j=1}^d\int_{\Omega}a_{ij}(x)D_iuD_j{v}\:\text{d}x+\int_{\Omega}c(x)u{v}\:\text{d}x,\quad u,v\in V\subset H^1,$$
where $a_{ij}(x)\in L^{\infty}$, $c(x)\in L^{\infty}$. Moreover, $a_{ij}(x)$ are symmetric and the following conditions hold:
\begin{align*}
	\sum_{i,j=1}^d a_{ij}(x)&\xi_i\xi_j\geq C|\xi|^2,&& \xi=(\xi_1,\ldots,\xi_n)\in\mathbb{R}^d,\text{ a.e. }x\in\Omega,&\\
	c(x)&\geq c_0>0,&&\mathrm{~a.e.~}x\in\Omega.&
\end{align*}
Then the operator $-A$ is a positive definite self-adjoint sectorial operator on $X$. According to \cite[Chapter 16]{ABsPEE}, when $V=H^1$, the domains of fractional powers of $A$ are specified as
\begin{equation*}
	\left\{\;
	\begin{aligned}
		&D(A^\theta)=H^{2\theta},\qquad \text{if}\:\: 0\leq\theta<\frac{3}{4},\\
		&D(A^\theta)=H_N^{2\theta},\qquad \text{if}\:\:\frac{3}{4}<\theta\leq 1,
	\end{aligned}\right.
\end{equation*}
with norm equivalence 
\begin{align}\label{NormEq1}
	C^{-1}\Vert u\Vert_{2\theta}\leq\Vert A^\theta u\Vert \leq C\Vert u\Vert _{2\theta}, \:\:u\in D(A^\theta).
\end{align}
While alternative choices of $V$ would lead to operators with different boundary conditions (Dirichlet, Robin, etc.), we restrict our analysis to the $H^1$-case for simplicity of presentation. 

Turning to the nonlinear term $f(u)$, a key step identified in \cite{ERK_FNM} for deriving fully discrete error estimates for a semilinear parabolic problem with quadratic nonlinearities was to establish the boundedness of its Fréchet derivatives in $D(A^{s/2})$. By leveraging the specific structure of $f(u)$, this was achieved using pointwise multiplication estimates in fractional Sobolev spaces. Building upon this methodology, the present study extends the analysis to general semilinear parabolic problems under the following smoothness assumption:

\begin{assumption}\label{assum:nf_smooth}
	Let $I=(a,b)$ where $-\infty\leq a<b\leq \infty$. We assume that the exact solution of \eqref{eq1} takes values almost everywhere in $I$ and the scalar function $f(\tau)$ is smooth on $I_\varepsilon=[a-\varepsilon,b+\varepsilon]$. 
\end{assumption}

This assumption ensures that the Nemytskii operator $f:L^\infty|_I\subset L^\infty \rightarrow L^\infty$ is Fréchet differentiable to an arbitrary order, and satisfies
\begin{equation}\label{Df_def}
	D_u^{(m)}f(u)(v_1,v_2,...,v_m) = f^{(m)}(u)\prod_{i=1}^m v_i, \quad m>0\text{, } u\in L^{\infty}|_I \text{, }v_i\in L^{\infty}, 
\end{equation} 
where $D_u^{(m)}f(u)$ is the $m$-th Fréchet derivative of $f(u)$, $f^{(m)}(\tau)$ is the $m$-th derivative of the scalar function $f(\tau)$, and the right-hand side is defined by classical pointwise multiplication. Unless otherwise specified, the composite operator $f(u)$ and the scalar function $f(\tau)$ use the same symbol $f$. 


According to \eqref{Df_def}, analyzing the boundedness of the Fréchet derivative of $f(u)$ is equivalent to estimating the pointwise multiplication on its right-hand side. To this end, we first define the relation set $\mathcal{B}$:  
\begin{equation*}\label{Set_B}
	\begin{aligned}
		\mathcal{B}=&\Big \{(s,s_1,s_2)\mid -\frac{3}{2}< s\leq s_i< \frac{3}{2}\:\: \text{for}\:\: i=1,2,  \\
		&\quad  s_1+s_2> \frac d2+s  ,\:\:\:s_1+s_2>0, \:\: s_1>0 \:\Big\},
	\end{aligned}
\end{equation*}
where the parameter bounds $(-3/2, 3/2)$  are imposed to facilitate the extension of Lemma \ref{le:multi} to its discrete version \eqref{mult_Xh}. Then we introduce the binary pointwise multiplication estimate from \cite[Lemma 1]{ERK_FNM}. 

\begin{lemma}\label{le:multi}
	Let $(s,s_1,s_2)\in \mathcal{B}$. Given $u\in D(A^{s_1/2})$, $v\in D(A^{s_2/2})\cap L^2$ and $uv\in L^2$, we have
	\begin{align*}
		\Vert A^{s/2}(u v)\Vert\leq C\Vert A^{s_1/2}u\Vert \Vert A^{s_2/2} v\Vert,
	\end{align*}
	where $C$ is independent of $u$ and $v$.
\end{lemma}

The $m$-ary case in \eqref{Df_def} can be generalized from the binary case. Additionally, the boundedness of the Nemytskii operator $f^{(m)}(u)$ is unknown, hence the following assumption is made.

\begin{assumption}\label{assum:bound_compos}
	For integer $m\geq 1$, there exist some increasing continuous functions $\phi_m(\cdot)$ and $\rho_m(\cdot)$ such that
	\begin{equation*}
		\begin{aligned}
			\Vert A^{\rho_m(s)/2} f^{(m)}(u)&\Vert \leq \phi_m(\Vert A^{s/2}u\Vert),  \qquad  u\in D(A^{s/2})|_I,\:\: 0\leq s\leq 2,
			\\
			\lim_{s\rightarrow\frac{d}{2}^+}&\rho_m(s)=\frac{d}{2},\quad \qquad \rho_m(s)=s\:\: \text{ for }\:\:\frac{d}{2}<s\leq 2.
		\end{aligned}
	\end{equation*}
	Here $C$ depends on $f^{(m)}(\tau)$ and $s$, but is independent of $u$.
\end{assumption}

The reasonableness of this assumption will be illustrated with examples at the end of this section. Under the above lemma and assumptions, we can obtain the local Lipschitz condition for $f(u)$. Selecting $s$, $s_1$ and $s_2$ such that $(s,\rho_1(s_1),s_2)\in \mathcal{B}$ and expanding $f(u)$ in a Taylor series at $v$ gives
\begin{equation}\label{Lips1}
	\begin{aligned}
		\Vert A^{s/2} (f(u)-f(v))\Vert &= \left\Vert A^{s/2}  \int_0^1 f'(v+\theta (u-v))(u-v) \text{d}\theta  \right\Vert
		\\ 
		&\leq C \int_{0}^{1} \left\Vert A^{\rho_1(s_1)/2}  f'(v+\theta (u-v))\right\Vert \text{d}\theta \: \Vert A^{s_2/2} (u-v)\Vert
		\\
		&\leq C \phi_1(\Vert A^{s_1/2}u\Vert+\Vert A^{s_1/2}v\Vert) \Vert A^{s_2/2} (u-v)\Vert,
	\end{aligned}
\end{equation}
where $u,v \in L^\infty|_I \cap D(A^{s_1/2}) \cap D(A^{s_2/2})$. In particular, \eqref{Lips1} implies the bound
\begin{equation}\label{Nf_bound}
	\Vert A^{s/2} f(u)\Vert\leq C \phi_1(\Vert A^{s_1/2}u\Vert) \Vert A^{s_2/2} u\Vert+C.
\end{equation}

For a semilinear parabolic system with the operator $A$ and a nonlinear term $f(u)$ which satisfies the local Lipschitz condition \eqref{Lips1} with $(0,\rho_1(\beta_1),\eta)\in \mathcal{B}$ ($\eta>\beta_1$), the local existence of solutions follows from \cite[Theorem 4.1]{ABsPEE}. Here, the parameter $\beta_1$ determines the range of initial value, and we hope it is as small as possible. Therefore, we make the following assumptions about the initial value.

\begin{assumption}\label{assum:iniData}
	Let the initial data $u_0\in D(A^{\gamma/2})$ with $\beta_1<\gamma<2$. Here the parameter $\beta_1$ satisfies $0\leq \beta_1<\frac{d}{2}$ and $\rho_1(\beta_1)\geq 0$. 
\end{assumption}

Thus, the semilinear parabolic problem \eqref{eq1} is rigorously established. By \cite[Theorem 4.2]{ABsPEE}, the solution of \eqref{eq1} has the following regularity
\begin{equation*}
	u\in \mathcal{C}((0,T];D(A))\cap \mathcal{C}([0,T];D(A^{\gamma/2}))\cap \mathcal{C}^1((0,T];X).
\end{equation*} 
It also holds that 
\begin{align}\label{bound_u}
	\|A^{s/2}u(t)\|\leq C&t^{\gamma/2-s/2}+C, \quad  0\leq s\leq 2, 
\end{align}
where $C$ is independent of $t$. 
Using the variation-of-constants formula, we obtain 
\begin{align}
	u(t)=S(t) u_0+\int_{0}^{t}S(t-\tau)f(u(\tau))\text{d}\tau,\qquad t\in[0,T].\label{VoCF1}
\end{align}

Next, we provide two examples of nonlinear terms to illustrate the reasonableness of the assumptions. 

\begin{example} \label{example:poly}
	$f(\tau)=\tau^m$ for $m\geq 2$. The function is smooth on $\mathbb{R}$. By \cite[Theorem 5.3.2.1]{SobFracNEMy} and the equivalence between spaces, we have
	$$
	\Vert f^{(k)}(u)\Vert_{s_{m-k}} \leq C\Vert u\Vert_s^{m-k }, \qquad
	0\leq k \leq m-1,\: u\in D(A^{s/2})|_I, $$
	where $s_{m-k}=s-(m-k-1)(\frac{d}{2}-s)$ for $d\max(0,\frac{1}{2}-\frac{1}{m-k}) < s <\frac{d}{2}$ or $s_{m-k}=s$ for $s>\frac{d}{2}$. When $k\geq m$, the derivative $f^{(k)}(u)$  becomes trivial. Due to the smoothness of $f(\tau)$, the homogeneous Neumann boundary conditions are preserved under nonlinear operations. Thus, combining norm equivalence \eqref{NormEq1}, Assumption \ref{assum:bound_compos} holds with $\rho_1(s)=\min(s-(m-2)(\frac{d}{2}-s),s)$. For Assumption \ref{assum:iniData}, the parameter $\beta_1=\frac{d(m-2)}{2(m-1)}$. A special case occurs for $u\in D(A^{s/2})\cap L^{\infty}$, $0<s<\frac{d}{2}$, where \cite[Theorem 5.3.2.4]{SobFracNEMy} yields $\rho_1(s)=s$. 
\end{example}

\begin{example}\label{example:ration}
	$f(\tau)=\tau^4/(1+\tau^2)$. Note that the function $f'(\tau)$ is infinitely differentiable and its derivatives are uniformly bounded on $\mathbb{R}$. By \cite[Section 5.3.6]{SobFracNEMy} and the equivalence between spaces, we have
	$$
	\Vert f^{(k)}(u)\Vert_{s'} \leq C \Vert u\Vert_s, \quad k\geq 0,\: u\in D(A^{s/2})|_I.$$
	where $s'=s$ for $0<s<1$ and $s>\frac{d}{2}$, or $s'=\frac{d}{2}/(\frac{d}{2}-s+1)$ for $1<s<\frac{d}{2}$. Similarly, it can be verified that Assumption \ref{assum:bound_compos} holds with $\rho_1(s)=s'$, and Assumption \ref{assum:iniData} is satisfied with $\beta_1 = 0$.
\end{example}

\section{Error analysis for spatial discretization scheme}\label{sec:spatiAnal}
While the primary focus of this paper is on the temporal discretization error, a practical convergence analysis must also incorporate spatial discretization. In this section, we discretize problem \eqref{eq1} in space using a linear Galerkin finite element method and analyze the resulting error. The analysis of the temporal discretization for the resulting semi-discrete system \eqref{eq2} is presented in the subsequent section.

\subsection{Spatial discretization scheme}
Consider a regular triangulation $\pi_h$ of the domain $\Omega$ with maximum element diameter $h$. We define $V_h \subset V$ as the finite-dimensional subspace consisting of continuous, piecewise linear functions over $\pi_h$, which serves as the approximation space $X_h = V_h$ for our spatial discretization scheme. The projection operator $P_h$ from $X$ to $X_h$ is defined as 
\begin{equation*}
	(P_h u,v_h)=(u,v_h),\quad \forall\: v_h\in X_h,\text{ for } u\in X.
\end{equation*}
The discrete operator $A_{h}:X_h\longrightarrow X_h$ is defined by
\begin{equation*}
	(-A_{h}u_h,v_h)=a(u_h,v_h),\quad\forall\: v_h\in X_h,\text{ for }u_h\in X_h.
\end{equation*}
Subsequently, we define the Ritz operator $R_{h}:V\longrightarrow X_h$, 
\begin{equation*}
	{a}(R_{h}u,v_h)={a}(u,v_h),\qquad \forall\: v_h\in X_h,\text{ for }u\in V.
\end{equation*}
Combining these components, we obtain the Galerkin finite element scheme for \eqref{eq2}: find $u^h(t)\in X_h$, such that
\begin{equation}\label{eq2}
	\left\{\;
	\begin{aligned} 
		&\frac{du^h(t)}{dt}=A_h u^h(t)+f_h(u^h(t)), \quad t\in (0,T],\\
		&u^h(0)=P_h u_0,
	\end{aligned}
	\right.
\end{equation}
where $f_h(u^h(t))=P_h (f(u^h(t)))$. 

Our error analysis is based on the semigroup technique, exploiting the fact that the operators $A$ and $A_h$ generate the analytic semigroups $S(t)=e^{tA}$ and $S_h(t)=e^{tA_h}$ on $X$ and $X_h$, respectively. The necessary estimates for these semigroups are given in the following lemma.
\begin{lemma}\label{le:semigroup}
	Let $\alpha,\alpha'\in \mathbb{R}$ and $0\leq \theta\leq 1$. Then the following estimates hold (see \cite{ABsPEE,SG_P}).
	\begin{align*}
		\Vert A^\alpha S(t)\Vert _{\mathcal{L}(X)} &\leq Ct^{-\alpha},& & t>0,\:\alpha\geq 0& \\
		\Vert A^{-\theta}(I-S(t))\Vert _{\mathcal{L}(X)}&\leq Ct^\theta,& & t\geq 0,&\\
		A^\alpha S(t)&=S(t) A^\alpha, & &\text{on}\enspace D(A^\alpha),&\\
		D(A^\alpha)&\subset D(A^{\alpha'}), & &\text{if}\enspace \alpha\geq\alpha'.&
	\end{align*}
	Furthermore, these estimates hold with a uniform constant $C$ (independent of $h$) when $A$ and $S(t)$ are replaced by their discrete versions $A_h$ and $S_h(t)$, respectively.
\end{lemma}

Then, we establish the relationship between the original problem \eqref{eq1} and the spatial discretization problem \eqref{eq2}. The resulting estimates are crucial for analyzing the spatial discretization error, and will also be used later for studying the temporal error.

\begin{lemma}\label{le:spaDis_property}
	For spatial semi-discretization of Problem \eqref{eq1}, the Galerkin finite element method \eqref{eq2} exhibits the following properties:
	\begin{align}
		\Vert A^\theta u_h\Vert &\leq C\Vert A^\theta_h u_h\Vert,&& u_h\in X_h\text{, }-1\leq \theta < {3}/{4},&\label{Auh}   
		\\
		\Vert A_h^\theta P_h u\Vert &\leq C\|A^\theta u\|,&& u\in D(A^\theta)\cap L^2 \text{, }-3/4< \theta\leq 1,&\label{AhPh}
		\\
		\Vert A_h^\theta u_h \Vert &\leq C h^{-2\theta}\Vert u_h\Vert,  && u_h\in X_h\text{, }0\leq \theta\leq 1,& \label{Ah_inver}  
		\\
		\|(1-R_{h})u\|_s &\leq C h^{r-s}\|A^{r/2}u\|, && u\in \:D(A^{r/2}),\:s\in[0,3/2),\:r\in[\max(1,s),2].&\label{I_Rh} 
	\end{align}
\end{lemma}

\begin{proof}
	The estimates \eqref{Auh}-\eqref{Ah_inver} follow directly from \cite[Lemma 3]{ERK_FNM} and \cite[Section 7.3]{ABsPEE}. The case $s\in[0,1]$ in estimate \eqref{I_Rh} is also covered by the above references, while the case $s>1$ requires additional proof. Given $s'\in (1,3/2)$. From \cite[Section 7.3.4]{ABsPEE}, we recall the fundamental estimates:
	\begin{align}
		&\Vert(1-R_h)u \Vert \leq Ch^{s'}\Vert A^{s'/2} u\Vert, && u\in D(A^{s'/2}),& \label{Rh_diif_1} 
		\\
		&\Vert(1-R_h)u \Vert \leq Ch^{2}\Vert A u\Vert, && u\in D(A),& \label{Rh_diif_2}
		\\
		&\Vert(1-R_h)u \Vert_{s'} \leq Ch^{2-s'}\Vert A u\Vert, && u\in D(A),& \label{Rh_diif_3} 
	\end{align}
	and 
	\begin{equation*}
		\Vert R_h u \Vert_1 \leq \Vert u\Vert_1 \text{ for } u\in V,\qquad  A_h R_h u = P_h A u \text{ for } u\in D(A).
	\end{equation*}
	The formulas above yield two key bounds:
	\begin{equation*}
		\Vert A_h^{1/2} R_h u\Vert \leq C \Vert A^{1/2}u \Vert \quad \text{and} \quad \Vert A_h R_h u\Vert \leq C \Vert A u \Vert.
	\end{equation*}
	Applying the Heinz-Kato inequality \cite[Theorem 2.31]{ABsPEE}, we obtain  
	\begin{equation*}
		\Vert A_h^{\theta} R_h u\Vert \leq C \Vert A^{\theta}u \Vert, \quad  1/2\leq \theta \leq 1.
	\end{equation*} 
	This implies the following estimate:
	\begin{equation}\label{Rh_diif_4}
		\Vert(1-R_h)u \Vert_{s'} \leq C\Vert A^{s'/2} u\Vert \quad \text{for } u\in D(A^{s'/2}).
	\end{equation}
	By combining inequalities \eqref{Rh_diif_1}-\eqref{Rh_diif_4} and applying operator interpolation theory \cite[Proposition 14.1.5]{TMToFEM}, we first interpolate the inequalities pairwise and then interpolate the resulting estimates to obtain the bound \eqref{I_Rh}. This completes the proof.
\end{proof}

\begin{remark}
	By combining \eqref{NormEq1}, \eqref{Auh} and \eqref{AhPh}, we establish the following norm equivalence relation on the discrete space $X_h$: 
	\begin{align}
		\|A_h^\theta \cdot\| \sim \|&A^\theta \cdot\| \sim \|\cdot\|_{2\theta}, \quad  0 \leq \theta < \frac{3}{4}.
		\label{NormEq2}
	\end{align}
	This norm equivalence and the semigroup properties in Lemma \ref{le:semigroup} are frequently used in subsequent analyses. To maintain conciseness, we will not explicitly reiterate them each time they are referenced.
\end{remark}

\subsection{Spatial error analysis} 
In our previous work \cite[Section 3.2]{ERK_FNM}, we first investigate the error associated with the linear part of problem \eqref{eq2}, then analyze the spatial discretization error estimates for local solutions. Finally, we extend the solution to the global domain using this local error estimate. This section requires the same argument to establish our spatial discretization results. We present only the key differences and main results. First, the Lipschitz condition of $f_h$ on the discrete space $X_h$ follows directly from \eqref{Lips1}, \eqref{Auh} and \eqref{AhPh}:
\begin{align}\label{Lips2}
	\Vert A_h^{s/2}P_h (f(u_h)-f(v_h))\Vert\leq C \phi_1(\Vert A_h^{\beta_1/2}u_h\Vert+\Vert A_h^{\beta_1/2}v_h\Vert) \Vert A_h^{s_2/2} (u_h-v_h)\Vert,
\end{align}
where $(s,\rho_1(\beta_1),s_2)\in \mathcal{B}$. In three dimensions, the parameter $\beta_1$ can exceed $1$, necessitating an extension of the parameter $s$ in the following lemma.

\begin{lemma}\label{le:lin_spa_err}
	Let $S(t)$ and $S_h(t)$ be the analytic semigroup generated by $A$ and $A_h$, respectively. For $s\in[0,3/2)$, $r\in [\max(s,1),2]$, and $\alpha\in[0,r]$, if $w_0\in D(A^{\alpha/2})$, then the following estimate holds:
	\begin{align*}
		\|(S(t)-S_h(t)P_h) w_0\|_s\leq Ch^{r-s}t^{-(r-\alpha)/2}\|w_0\|_\alpha,\qquad t\in (0,T].
	\end{align*}
\end{lemma}

The proof follows immediately from estimate \eqref{I_Rh}. Subsequently, based on the lemma above and the Lipschitz condition \eqref{Lips2}, we can establish the well-posedness of the discrete problem and derive the spatial discretization error estimates in the following theorem.

\begin{theorem}\label{thm:space_err}
	Let $u(t)$ be the solution of \eqref{eq1}. Assume that Assumptions \ref{assum:nf_smooth}-\ref{assum:iniData} are fulfilled. Then for semidiscrete problem \eqref{eq2}, there exists $h'>0$ such that for $h<h'$, the unique solution $u^h(t)$  exists  on $[0,T]$. Moreover, the following estimates hold:
	\begin{align}
		\Vert u(t)-u^h(t)\Vert_{\mu} &\leq Ct^{-1+\gamma/2}h^{2-\mu},  
		&&0\leq \mu \leq\hat{\gamma},& \label{SD_G_1}
		\\
		\Vert u(t)-u^h(t)\Vert_{\mu} &\leq Ch^{\gamma-\mu},  
		&&0\leq \mu \leq1,\:1\leq \gamma \leq 2,& \label{SD_G_2}
		\\
		\Vert A_h^s u^h(t)\Vert &\leq {Ct^{{\gamma}/2-s/2}+C}, 
		&&0\leq s\leq 2.&\label{bound_uh}
	\end{align}
	Here, $\hat{\gamma}=\min(3/2-\varepsilon,\gamma)$, and this notation is adopted hereafter.
\end{theorem}

\section{Fully discrete scheme and Error representations}\label{sec:NumSche_ErrExpa}
Building on the spatial discretization error analysis in Section \ref{sec:spatiAnal}, this section presents the fully discrete scheme and proceeds to examine the temporal error. Following the convergence analysis approach in \cite{luan12stiff,luan14stiff}, we decompose the error into two components: stability and local error. While the stability of the scheme can be straightforwardly derived from the Lipschitz condition \eqref{Lips2}, a rigorous analysis of the local error demands considerable effort, relying on a detailed Taylor expansion.

\subsection{Fully discrete scheme}
For the temporal discretization of semidiscrete problem \eqref{eq2}, we consider a class of $\kappa$-stage explicit exponential Runge-Kutta methods (EERK) with constant stepsize $\delta=T/N$.
\begin{equation}\label{NumSche1}
	\begin{aligned}
		U^h_{ni}&=S_h(c_i \delta)u^h_n+\delta\sum_{j=1}^{i-1}a_{ij}(\delta A_h)f_h(U^h_{nj}),\quad 1\leq i\leq \kappa,\\
		u^h_{n+1}&=S_h(\delta)u^h_n+\delta\sum^\kappa_{i=1} b_i(\delta A_h)f_h(U^h_{ni}).
	\end{aligned}
\end{equation}
The inner stages $U^h_{ni}$ serve as approximations to $u^h(t_n + c_i \delta)$, while the numerical solution $u^h_{n+1}$ approximates the solution $u^h(t)$ at time $t_{n+1}$. The coefficients $a_{ij}(\delta A_h)$ and $b_i(\delta A_h)$ are chosen as linear combinations of the exponential functions $\varphi_k(c_i\delta A_h)$ and $\varphi_k(\delta A_h)$, respectively. These functions are given by
$$
\varphi_{0}(z)=e^{z},\quad\varphi_{k}(z)=\int_{0}^{1}e^{(1-\theta)z}\frac{\theta^{k-1}}{(k-1)!}\text{d}\theta,\quad k\geq 1, $$
and thus satisfy the recurrence relation
$$
\varphi_{k+1}(z)=\frac{\varphi_k(z)-\varphi_k(0)}{z},\quad k\geq0.$$
Similar to Lemma \ref{le:semigroup}, for $0\leq\alpha<1$, we get
$$
\Vert A_h^\alpha b_i(tA_h)\Vert_{\mathcal{L}(X_h)}+\Vert A_h^\alpha a_{ij}(tA_h)\Vert_{\mathcal{L}(X_h)} \leq Ct^{-\alpha},\quad t>0.$$

To keep the notation concise, we introduce the abbreviations $a_{ij}=a_{ij}(\delta A_h)$, $b_{i}=b_{i}(\delta A_h)$, $\varphi_i=\varphi_{i}(\delta A_h)$, and $\varphi_{ji}=\varphi_{j}(c_i\delta A_h)$. For a $p$th-order scheme with $s$ stages, the coefficients can be compactly represented by a Butcher tableau. 
%
For higher accuracy, we consider the classical three-stage third-order exponential Runge-Kutta method (EERK3), originally proposed in \cite{ERK1}. The corresponding coefficients are detailed in \eqref{EERK3}. While this scheme provides superior accuracy for smooth problems, our primary interest lies in its convergence behavior and order reduction phenomena when applied to equations with nonsmooth initial data.
\begin{equation}\label{EERK3}
	\renewcommand{\arraystretch}{1.2}
	\begin{array}{c|ccc}
		0 & & & \\
		1/2 & \frac{1}{2}\varphi_{1,2} & & \\
		2/3 & \frac{2}{3}\varphi_{1,3} - \frac{8}{9}\varphi_{2,3} & \quad\frac{8}{9}\varphi_{2,3} & \\
		\hline
		& \varphi_1 - \frac{3}{2}\varphi_2 & 0 & \quad\frac{3}{2}\varphi_2
	\end{array}
\end{equation}

Furthermore, we define $g(t)=f_h(u^h(t))$, $F(t)=A_h u^h(t)+g(t)$, and $\tilde{u}^h_{n}=u^h(t_{n})$. In the single-step numerical scheme \eqref{NumSche1}, the local approximations $\hat{u}^h_{n+1}$ and $\hat{U}^h_{ni}$ are obtained after one step starting from $\tilde{u}^h_n$. The temporal error is then split as follows:
\begin{align*}
	e_{n+1}=\tilde{u}^h_{n+1}-u^h_{n+1}=\tilde{u}^h_{n+1}-\hat{u}^h_{n+1}+\hat{u}^h_{n+1}-u^h_{n+1}=:\tilde{e}_{n+1}+\hat{e}_{n+1},
\end{align*}
where $\tilde{e}_{n+1}$ and $\hat{e}_{n+1}$ represent the local error and stability error of the numerical scheme, respectively.

\subsection{Stability}
\begin{proposition}\label{pro:stabi}
	Under Assumptions \ref{assum:nf_smooth}-\ref{assum:iniData}, there exist operators $\mathcal{N}(e_n)$ on $X_h$ such that
	\begin{align}\label{stabi_expan}
		\hat{e}_{n+1}=S_h(\delta)e_n+\delta \mathcal{N}(e_n), \quad 0\leq n\leq N-1.
	\end{align} 
	Furthermore, for $0\leq \mu\leq \hat{\gamma}$, there exist $r \in\big(0,\frac{d}{2}\big)$ with $r+\mu<2$ and $(-r,\rho_1(\beta_1),\mu) \in\mathcal{B}$ such that
	\begin{align}
		\Vert A_h^{-r/2}\mathcal{N}(e_n)\Vert\leq \Vert A_h^{\mu/2} e_n\Vert,\label{stabi_est1}\\
		\Vert \mathcal{N}(e_n)\Vert\leq  \delta^{-r/2}\Vert A_h^{\mu/2} e_n\Vert, \label{stabi_est2}
	\end{align}
\end{proposition}

\begin{proof}
	Define $\mathcal{N}_{ni}(e_n) = f_h(\widehat{U}^h_{ni})-f_h(U^h_{ni})$. Its estimation relies on the Lipschitz condition \eqref{Lips2}, necessitating boundedness of $\Vert u^h_n \Vert_{\beta_1}$. Following \cite[Theorem 2]{ERK_FNM}, this boundedness can be transformed into a convergence analysis and established via mathematical induction. For simplicity, we assume that the following boundedness property holds throughout this paper:
	$$
	\Vert u^h_{n}\Vert_{\beta_1} + \Vert {U}^h_{ni}\Vert_{\beta_1}<C,\quad 0\leq n\leq N. $$
	Then we get
	\begin{align*}
		\Vert A_h^{-r/2} \mathcal{N}_{ni}(e_n)\Vert \leq \Vert A_h^{\mu/2}\widehat{E}_{ni} \Vert,
	\end{align*}
	where $\widehat{E}_{ni}=\widehat{U}^h_{ni}-U^h_{ni}$. Based on the numerical scheme \eqref{NumSche1}, we obtain
	\begin{align*}
		\hat{e}_{n+1}&=S_h(\delta)e_n + \delta \sum_{i=1}^{\kappa} b_i \mathcal{N}_{ni}(e_n),\\
		\widehat{E}_{ni}&=S_h(c_i\delta)e_n + \delta \sum_{j=1}^{i-1} a_{ij} \mathcal{N}_{nj}(e_n).
	\end{align*}
	Therefore, both estimates \eqref{stabi_est1} and \eqref{stabi_est2} follow at once from the inequality $\Vert A_h^{-r/2} \mathcal{N}_{ni}(e_n) \Vert \leq \Vert A_h^{\mu/2} e_n \Vert $. We prove this inequality by induction. For the base case $i=1$, $\widehat{E}_{n1}=e_n$, then 
	$\Vert A_h^{-r/2} \mathcal{N}_{n1}(e_n) \Vert \leq \Vert A_h^{\mu/2} e_n \Vert $. Now, assume the inequality holds for $m=2,...,l-1$. For the case $m=l\leq \kappa$, we have
	\begin{align*}
		\Vert A_h^{-r/2} \mathcal{N}_{nl}(e_n)\Vert 
		&\leq \Vert A_h^{\mu/2}\widehat{E}_{nl} \Vert\\
		&\leq \Vert A_h^{\mu/2} e_n \Vert +  \sum_{j=1}^{l-1} \Vert \delta A_h^{\mu/2+r/2} a_{i,j} \Vert_{\mathcal{L}(X_h)} \Vert A_h^{-r/2} \mathcal{N}_{nj}(e_n) \Vert\\
		&\leq \Vert A_h^{\mu/2} e_n \Vert.
	\end{align*}
	This completes the proof.
\end{proof}

By repeatedly applying \eqref{stabi_expan}, there exist $r_1\in (0,\frac{d}{2})$ such that
\begin{equation}\label{err_L2Spli}
	\begin{aligned}
		\Vert e_{n+1}\Vert &=\Vert S_h(\delta)e_n+\delta \mathcal{N}(e_n)+\tilde{e}_{n+1}\Vert
		\\
		&=\left\Vert \sum_{k=1}^{n+1}S_h(t_{n+1-k})\tilde{e}_{k}+ \sum_{k=1}^{n}\delta S_h(t_{n-k})\mathcal{N}(e_k)\right\Vert
		\\
		&\leq \left\Vert \sum_{k=1}^{n+1} S_h(t_{n+1-k})\tilde{e}_{k}\right\Vert +  \delta\sum_{k=1}^{n-1}  t_{n-k}^{-r_1/2} \Vert e_k \Vert + \delta^{1-r_1/2} \Vert e_n\Vert.
	\end{aligned}
\end{equation}
To estimate the first term on the right-hand side, it is necessary to describe the local error.

\subsection{Taylor expansion of the local error}
In this subsection, we apply the approach from \cite{luan12stiff} to expand the local error. Upon expansion, if the numerical method satisfies the corresponding order conditions, the local error will contain some higher-order remainder terms. The primary challenge in nonsmooth error analysis lies in accurately estimating these terms. We therefore analyze the expansion process to examine their composition. Like \eqref{Df_def}, the nonlinear term $f_h:X_h|_I\subset X_h\rightarrow X_h$ is Fréchet differentiable of arbitrary order, and satisfies
\begin{equation}\label{Dfh_def}
	D_u^{(m)}f_h(u_h)(v^h_1,v^h_2,...,v^h_m) = P_h(f^{(m)}(u_h)\prod_{i=1}^m v^h_i),\quad m>0\text{, }u_h, v_i^h\in X_h.
\end{equation} 
Expressing the exact solution $u^h(t)$ at time $t_{n+1}$ by the variation-of-constants formula,
$$
\tilde{u}^h_{n+1}=S_h(\delta)\tilde{u}^h_n+\delta\int_{0}^{1}S_h((1-\theta)\delta)g(t_n+\theta \delta)\text{d}\theta, $$
then expanding $g(t_n +\theta \delta)$ in a Taylor series at $t_n$ gives
\begin{equation}\label{Expan_exa}
	\begin{aligned}
		\tilde{u}^h_{n+1}=&\:\tilde{u}^h_n+\delta \varphi_1F(t_n)+\sum_{i=2}^{p}\delta^i \varphi_i g^{(i-1)}(t_n)\\
		&+\delta^{p+1} \int_{0}^{1}S_h((1-\theta)\delta) \frac{\theta^p}{(p-1)!} \int_{0}^{1}(1-s)^{p-1} g^{(p)}(t_n+s \theta \delta)\text{d}s \text{d}\theta.
	\end{aligned}
\end{equation} 

After the expansion of the exact solution $\tilde{u}^h_{n+1}$, we turn to expand the local numerical solution $\hat{u}^h_{n+1}$. It is desired that the numerical scheme preserves the equilibrium. This preservation property can be guaranteed by requiring
\begin{equation}\label{orderCond_1}
	\sum_{j=1}^{i-1}a_{ij}=c_i\varphi_1,\:\: i=1,\ldots,\kappa,\quad\sum_{i=1}^\kappa b_i=\varphi_1.
\end{equation}
Taking into account the above conditions, we reformulate the EERK method \eqref{NumSche1} and consider one step with initial value $\widetilde{u}^h_n$.
\begin{equation}\label{NumSche2}
	\begin{aligned}
		\widehat{U}^h_{ni}&=\tilde{u}^h_{n}+c_{i}\delta\varphi_{1i}F(t_n)+\delta\sum_{j=2}^{i-1}a_{ij}\widehat{D}_{nj},\\
		\hat{u}^h_{n+1}&=\tilde{u}^h_{n}+\delta\varphi_{1}F(t_n)+\delta\sum_{i=2}^{\kappa}b_{i}\widehat{D}_{ni},
	\end{aligned}
\end{equation}
with  
$
\widehat{D}_{ni}=f_h(\widehat{U}^h_{ni})-f_h(\tilde{u}^h_n).$
Expanding $\widehat{D}_{ni}$ in Taylor series at $\widetilde{u}^h_n$, we obtain
\begin{equation}\label{recur1_mid}
	\begin{aligned}
		\widehat{D}_{ni}
		&=\delta \int_0^1 P_h( f'(\tilde{u}^h_n+\theta \delta V_i)V_i) \text{d}\theta \\
		&= \sum_{k=1}^{m}\delta^k\frac{1}{k!}P_h(f^{(k)}(\tilde{u}^h_n)(V_i)^k)+ \delta^{m+1}\int_0^1 \frac{(1-\theta)^m}{m!}P_h(f^{(m+1)}(\tilde{u}^h_n+\theta \delta V_i)(V_i)^{m+1}) \text{d}\theta,
	\end{aligned}
\end{equation}
with
\begin{equation}\label{recur1_mid2}
	V_{i}=\frac{1}{\delta}\left(\widehat{U}^h_{ni}-\tilde{u}^h_{n}\right)=c_{i}\varphi_{1}(c_{i}\delta A_h)F(t_n)+\sum_{j=2}^{i-1}a_{ij}\widehat{D}_{nj}.
\end{equation}
The first term of right hand of above equality can be expanded as follow.
\begin{equation}\label{recur1_end}
	\begin{aligned}
		\varphi_1(c_i \delta A_h)F(t_n) &= \varphi_1(c_i \delta A_h) (\tilde{u}^h_n)'\\
		&= (\tilde{u}^h_n)' + c_i \delta \varphi_2(c_i \delta A_h) ( (\tilde{u}^h_n)''-g'(t_n)) \\
		&= (\tilde{u}^h_n)'+\frac{c_i \delta}{2!}\mathbf{X}_i + c_i^2 \delta^2 \varphi_3(c_i \delta A_h)((\tilde{u}^h_n)^{(3)}-g''(t_n)),
	\end{aligned}
\end{equation}
where $\mathbf{X}_i=(\tilde{u}^h_n)^{\prime\prime}-2!{\varphi}_2(c_i\delta A_h)g'(t_n)$ and the equality can be expanded infinitely. A finite number of nested iterations are performed among \eqref{recur1_mid}, \eqref{recur1_mid2}, and \eqref{recur1_end} (the number of iterations depends on the stage $\kappa$ of scheme \eqref{NumSche2}), and then substituted into \eqref{NumSche2} to obtain the expansion of the numerical solution. Subtracting \eqref{Expan_exa} yields the local error, from which the order conditions of the method can be established (see, e.g., \cite{luan12stiff,luan14stiff}). Below is an example of the expansion with a fifth-order remainder term, where ``order'' denotes the power of the step size $\delta$.
\begin{equation*}
	\begin{aligned}
		\tilde{e}_{n+1} &=\delta^2 \psi_2 g'(t_n)+\delta^3\psi_3 g''(t_n)
		+\delta^4\psi_4 g^{(3)}(t_n)+ \textbf{R}_4(t_n) + \textbf{O}_{5}(t_n)
	\end{aligned}
\end{equation*}
with
\begin{align*}
	\textbf{R}_4(t_n) &=\delta^{3}\sum_{i=2}^{\kappa}b_{i}P_h(f^{\prime}(\tilde{u}^h_{n})\psi_{2,i}g'(t_n))+\delta^{4}\sum_{i=2}^{\kappa}b_{i}P_h(f^{\prime}(\tilde{u}^h_{n})\psi_{3,i}g''(t_n))\\
	&\qquad +\delta^{4}\sum_{i=2}^{\kappa}b_{i}P_h(f^{\prime}(\tilde{u}^h_{n})\sum_{j=2}^{i-1}a_{ij}P_h(f^{\prime}(\tilde{u}^h_{n})\psi_{2,j}g'(t_n)))+\delta^{4}\sum_{i=2}^{\kappa}b_{i}c_{i}P_h(f^{\prime\prime}(\tilde{u}^h_{n})(\tilde{u}^h_{n})^{\prime}\psi_{2,i}g'(t_n)),
\end{align*}
where $\psi_j=\sum_{i=2}^\kappa b_i\frac{c_i^{j-1}}{(j-1)!}-\varphi_j$, $\psi_{ji}=\sum_{k=2}^{i-1}a_{ik}\frac{c_k^{j-1}}{(j-1)!}-c_i^j\varphi_{ji}$.
\begin{definition}
	In general, we say that an EERK method \eqref{NumSche1} is of order $p$ with $p\geq 2$ if it fulfills the order conditions in \cite{luan12stiff,luan14stiff,ERK1} up to order $p-1$ and $\psi_p(0) = 0$. Furthermore, the remaining conditions of order $p$ are satisfied in a weaker form, using $b_i(0)$ in place of $b_i(\delta A_h)$ for $2 \leq i \leq \kappa$.
\end{definition}
For a $p$th order EERK method, its local error expansion is given by:
\begin{equation}\label{Err_localReprens}
	\begin{aligned}
		\tilde{e}_{k+1}&=\delta^p\big(\psi_p(\delta A_h)-\psi_p(0)\big)g^{(p)}(t_k)+\delta^p\sum_{i=2}^\kappa\big(b_i(\delta A_h)-b_i(0)\big)\mathbf{Q}_p(t_{k})+\mathbf{O}_{p+1}(t_{k})\\
		&=:\tilde{e}_{k+1}^{(p)}+\mathbf{O}_{p+1}(t_k).
	\end{aligned}
\end{equation}
Here $\textbf{Q}_p(t_k)$ denotes the terms multiplying $\delta^p\sum_{i=2}^{\kappa}$ in expansion and $\mathbf{O}_{p+1}(t_k)$ is the higher-order remainder.
\begin{remark}\label{challenge_1}
	Conventional error analysis assumes that the derivatives above in the local error expansion are bounded on the underlying space $X$, which may not hold in non-smooth settings. This presents a key challenge: establishing the boundedness of the Fréchet derivatives of $f_h(u_h)$. The boundedness for $f$ was established in Section \ref{sec:AnaFrame} (Assumptions \ref{assum:nf_smooth}, \ref{assum:bound_compos}),  the discrete case requires a further step. In Section \ref{sec:ErrEst_KeyEst}, we combine with Lemma \ref{le:spaDis_property} to generalize the result, 
	which is simplified to Assumption \ref{ass:Frech_dominant}.
\end{remark}

\begin{remark}\label{challenge_2}
	Estimating the local error \eqref{Err_localReprens} faces another challenge:   determining the composition of $\textbf{Q}_p(t_k)$ and $\textbf{O}_{p+1}(t_k)$ in the remainder term, this is the task of the present subsection. Observing the expansion example $\textbf{R}_4(t_n)$ above, the expansion terms consist of step sizes, derivatives of $u^h(t_n)$ and $f_h(u_h)$, and exponential functions. The uniformly bounded exponential operators can be absorbed into the generic constants. A connection exists between the powers of the step size and the orders of derivatives in the terms, enabling an equivalent estimate. Finally, these estimates exhibit a singularity at $t=0$, which cancels out the corresponding powers of the step size. This results in a reduction of the method's order. 
\end{remark}

The expansion \eqref{Err_localReprens} leaves some higher-order terms, which satisfy the following pattern: 
\begin{enumerate}
	\item Based on the expected order of the method, we can control the degree of expansion in \eqref{Expan_exa}, \eqref{recur1_mid}, and \eqref{recur1_end}; for instance, in the error expansion of a $p$th-order method, the highest derivatives involved are $f^{(p)}(u_h)$ and $D_t^p u^h(t)$.
	\item Since the estimates for $S_h(\delta)$ and $\varphi_i(\delta A_h)$ are identical under the $\Vert A_h^{s/2}\cdot\Vert$ norm $(-2<s<2)$, we can equivalently substitute the latter with the former. The same applies to the exponential functions generated by $\varphi_i$.
	\item The boundedness of $f^{(m+1)}(\tilde{u}^h_n+\theta \delta V_i)$ is identical to that of $f^{(m+1)}(\tilde{u}^h_n)$, and the integral of the remaining terms in \eqref{recur1_mid} is bounded; thus, the former can be replaced by the latter.
\end{enumerate}
Below is an example of an equivalent substitution.
\begin{equation*}
\begin{aligned}
 &\delta^{4}\sum_{i=2}^{\kappa}b_{i}c_{i}\int_0^1 (1-\theta)P_h(f^{(2)}(\tilde{u}^h_n+\theta \delta V_i)(\tilde{u}^h_{n})^{\prime}\psi_{2,i}g'(t_n))\text{d}\theta \\
 \sim\  & 
\delta^{4}S_h(\delta)P_h(f^{(2)}(\tilde{u}^h_{n})(\tilde{u}^h_{n})^{\prime}S_h(\delta)g'(t_n)) 
\end{aligned}
\end{equation*}
It remains to estimate the derivative term on the right-hand side. Clearly, the orders of $f^{(i)}(\tilde{u}^h)$ and $(\tilde{u}^h_{n})^{(j)}$ exhibit a specific pattern determined by the powers of $\delta$. To characterize this pattern, we introduce a family of function sets, $\mathcal{K}_m^p$, where the index $m$ indicates the power of $\delta^m$ that pairs with the corresponding derivative. Denote $\mathcal{K}^{p}_1=\{D_t u^h(t)\}$ and $\mathcal{K}'=\{ D_t^i u^h(t) \:|\: i\geq 1\}$. 
The set $\mathcal{K}_{m+1}^p$ are constructed in a recursive way. For $1\leq m\leq p-1$, 
\begin{equation*}
	\begin{aligned}
		\mathcal{K}^{{p}}_{m+1}&=\{D_t^{m+1} u^h(t)\}\cup \{ S_h(\delta) P_h (f^{(j)}(u^h(t))\Pi_{i=1}^j v^h_i(t)) \big|\\
		&\qquad 1\leq j\leq {p},\: v^h_{i}(t)\in \mathcal{K}_{k_i}^{p},\:\sum_{i=1}^{j}k_i=m\}.
	\end{aligned}
\end{equation*}
For $m\geq {p}$, 
\begin{equation*}
	\begin{aligned}
		\mathcal{K}^{p}_{m+1}= \{ S_h(\delta) P_h(f^{(j)}(u^h(t))\Pi_{i=1}^j v^h_i(t)) \big|
		\:1\leq j\leq p,\: v^h_{i}(t)\in \mathcal{K}_{k_i}^{p},\:\sum_{i=1}^{j}k_i=m
		\}.
	\end{aligned}
\end{equation*}
Define $\text{span}(\mathcal{K}_m^p)$ as the vector space generated by the function set $\mathcal{K}_m^p$. Then there exist $V_i(t)\in \text{span}(\mathcal{K}_i^p)$ for $i\geq p$, such that $\textbf{Q}_p(t_k)$ and  $\mathbf{O}_{p+1}(t_k)$  can be equivalently represented by $V_p(t_k)$ and $\delta^{p+1}V_{p+1}(t_k)+\delta^{p+2}V_{p+2}(t_k)+...$, respectively.

\section{Error estimate}\label{sec:ErrEst}
In the previous section, the time discretization error was separated into two parts, resulting in \eqref{err_L2Spli}. It remains to estimate the local error \eqref{Err_localReprens}, which will be referred to as a $p$th-order expansion. For higher-order expansions, the analysis is complicated by the involvement of higher-order Fréchet derivatives of $f_h$. Therefore, we restrict our presentation to the second- and third-order cases and set $\hat{p}=2$ or $3$. The procedure for higher orders is analogous. Note that higher-order methods inherently satisfy lower-order conditions, thus the corresponding expansions remain valid, and the final error estimates are still applicable. For a $\hat{p}$th-order expansion, the terms $\textbf{Q}_{\hat{p}}(t_k)$ and $\textbf{O}_{\hat{p}+1}(t_k)$ can be equivalently represented by a linear combination of the elements in $\mathcal{K}^{\hat{p}}_{m}$. To estimate the equivalent terms in $\mathcal{K}_m^{\hat{p}}$, it is necessary to analyze the boundedness of the Fréchet derivative of $f_h$. Following this derivation, the result is simplified to Assumption \ref{ass:Frech_dominant}. Then the estimates concerning the equivalent terms are provided in Proposition \ref{pro:Km1}. Substituting these into \eqref{err_L2Spli} yields the temporal error in Theorem \ref{theo:tempErr1}.

\subsection{Boundedness of the Fréchet derivatives of $f_h(u^h(t))$}\label{sec:ErrEst_KeyEst}

Combining Lemma \ref{le:multi} with \eqref{Auh} and \eqref{AhPh} yields its discrete counterpart.
\begin{align}\label{mult_Xh}
	\Vert A_h^{s/2}P_h(v_1 v^h_2) \Vert\leq C\Vert A^{s_1/2} v_1\Vert \Vert A_h^{s_2/2}v^h_2\Vert, 
\end{align}
where $(s,s_1,s_2)\in \mathcal{B}$. Note that if $u$ and $v$ satisfy homogeneous Neumann conditions, so does their product $uv$. Moreover, when $s>\frac{3}{2}$, the fractional Sobolev space is an algebra. Consequently, the parameter range $\mathcal{B}$ in Lemma \ref{le:multi} can be extended. For the discrete version $P_h(f'(u^h(t))v^h)$, the numerical solution $u^h(t)$ lacks sufficient regularity for a direct estimation of $f'(u^h(t))$ in the norm $\Vert A^{s_1/2}\cdot \Vert$. To overcome this, one may substitute $f'(u^h(t))$ with $f'(u(t))$, and the resulting error can be controlled by applying the inverse inequality \eqref{Ah_inver}.

\begin{lemma}\label{le:mult_algebra}
	For $\frac{3}{{2}}<r\leq 2$, $v_1\in L^2$ and $v^h_2\in X_h$, if there exist $w_1\in D(A^{r/2})$ such that $\Vert w_1-v_1 \Vert_\varepsilon \leq C h^{r-\varepsilon}$, then the following estimates hold.
	\begin{align}
		\Vert A_h^{(r-\varepsilon)/2}P_h(v_1 v^h_2) \Vert&\leq C(\Vert A^{r/2} w_1\Vert+1) \Vert A_h^{r/2}v^h_2\Vert, \label{mult_algebra}
		\\
		\Vert A_h^{-r/2}P_h(v_1 v^h_2) \Vert &\leq C(\Vert A^{r/2} w_1\Vert+1) \Vert A_h^{(-r+\varepsilon)/2}v^h_2\Vert.
		\label{mult_algebra_dual}
	\end{align}
\end{lemma}
\begin{proof}
	The function $w_1$ serves as the smooth counterpart of $v_1$, and we now seek to construct an analogous smooth counterpart for $v_2^h$. To this end, we select $w_2\in D(A)$ such that $A_h v_2^h=A w_2$, which implies that $R_h w_2=v_2^h$. By \eqref{Auh} and the definition of $w_2$, we obtain the estimate
	$$
	\Vert A^{r/2} w_2\Vert \leq \Vert A^{r/2-1} A_h v_2^h \Vert \leq \Vert A_h^{r/2} v_2^h \Vert.
	$$ 
	Combining Lemma \ref{le:multi} and \ref{le:spaDis_property}, it holds that 
	\begin{align*}
		\Vert A_h^{(r-\varepsilon)/2}P_h(v_1 v^h_2) \Vert &\leq \Vert A_h^{(r-\varepsilon)/2}P_h[(v_1-w_1+w_1)(v^h_2-w_2+w_2)] \Vert
		\\
		&\leq \Vert A_h^{(r-\varepsilon)/2}P_h((v_1-w_1)(v^h_2-w_2))\Vert +\Vert A_h^{(r-\varepsilon)/2}P_h((v_1-w_1)w_2)\Vert
		\\
		&\quad +\Vert A_h^{(r-\varepsilon)/2}P_h (w_1(v_2^h-w_2))\Vert+\Vert A_h^{(r-\varepsilon)/2}P_h (w_1w_2) \Vert
		\\
		&\leq C h^{\varepsilon-r}\Vert v_1-w_1 \Vert_{\varepsilon} \Vert (1-R_h) w_2\Vert_{d/2-\varepsilon} +  
		C h^{\varepsilon-r}\Vert v_1-w_1\Vert \Vert A^{r/2} w_2\Vert 
		\\
		&\quad +C h^{\varepsilon-r} \Vert A^{r/2} w_1 \Vert \Vert (1-R_h)w_2 \Vert + C \Vert A^{r/2} w_1 \Vert  \Vert A^{r/2} w_2 \Vert
		\\
		&\leq C h^{\varepsilon-r+r-\varepsilon+r-d/2+\varepsilon} \Vert A^{r/2}w_2 \Vert + C h^{\varepsilon-r+r-\varepsilon} \Vert A^{r/2} w_2\Vert 
		\\
		&\quad+ C h^{\varepsilon-r+r} \Vert A^{r/2} w_1 \Vert \Vert A^{r/2} w_2 \Vert  + C\Vert 
		A^{r/2} w_1 \Vert \Vert A^{r/2} w_2 \Vert
		\\
		&\leq  C(\Vert A^{r/2} w_1\Vert+1) \Vert A_h^{r/2}v^h_2\Vert.
	\end{align*}
	Thus, \eqref{mult_algebra} is proved, and \eqref{mult_algebra_dual} follows by duality argument from \eqref{mult_algebra}.
	\begin{align*}
		&\quad \Vert A_h^{-r/2} P_h (v_1 v_2^h)\Vert = \sup_{w^h\in X_h} \frac{|(P_h(v_1 v_2^h), w^h)|}{\Vert A^{r/2} w^h \Vert}  =\sup_{w^h\in X_h} \frac{|(v_2^h,P_h(v_1 w^h))|}{\Vert A_h^{r/2} w^h \Vert} \\
		&\leq \sup_{w^h\in X_h} \frac{\Vert A^{(\varepsilon-r)/2}v_2^h \Vert \Vert A^{(r-\varepsilon)/2}(v_1 w^h)\Vert}{\Vert A^{r/2} w^h \Vert} 
		\leq (\Vert A^{r/2}w_1\Vert+1) \Vert A^{(\varepsilon-r)/2}v_2^h\Vert.
	\end{align*}
	This completes the proof.
\end{proof}
Similarly, we can extend the above binary pointwise multiplication estimates to the m-ary case, as shown below.
\begin{align}
	\Vert A_h^{(r-\varepsilon)/2}P_h(v_1 v^h_2 ... v_m^h) \Vert & \leq C(\Vert A^{r/2} w_1\Vert+1) \Vert A_h^{r/2}v^h_2\Vert ... \Vert A_h^{r/2}v^h_m\Vert.
	\label{mult_Xh_m1}
	\\
	\Vert A_h^{-r/2}P_h(v_1 v^h_2 ... v_m^h) \Vert  \leq C(\Vert &A^{r/2} w_1\Vert+1) \Vert A_h^{r/2}v^h_2\Vert ... \Vert A_h^{r/2}v^h_{m-1}\Vert \Vert A_h^{(-r+\varepsilon)/2}v^h_m\Vert.\label{mult_Xh_m2}
\end{align}
where $v_i^h\in X_h$ for $i=2,3,..,m.$

We now proceed to estimate the Fréchet derivatives of $f_h(u^h)$. Define a shift for $s$: 
\begin{equation}\label{eq:sigma}
	\sigma(s) = \max(s+d/2-\rho_1(\gamma),-\rho_1(\gamma),s)+\varepsilon, \quad -2\leq s<\rho_1(\gamma).
\end{equation}
Here, the inequality $\sigma(s) - s < 2$ holds, which will be frequently employed in subsequent arguments. We want to show that
\begin{align}\label{frech_first}
	\Vert A_h^{s/2}P_h (f'(u^h(t))v^h) \Vert\leq C  \Vert A_h^{\sigma(s)/2}v^h \Vert, \qquad -2\leq s<\rho_1(\gamma).
\end{align} 
For the case $|s|<{3}/{2}$ and $\gamma<3/2$ where $(s,\rho_1(\gamma),\sigma(s))\in \mathcal{B}$, estimate \eqref{frech_first} follows directly from \eqref{bound_uh} and \eqref{mult_Xh}. Note that through norm embedding, estimate \eqref{mult_Xh} remains valid for $(s-c_1, \rho_1(\gamma)+c_2, \sigma(s)+c_3)$ even outside the set $\mathcal{B}$ for any $c_1,c_2,c_3>0$. We will omit this special case in what follows. When $|s|>3/2$ and $\gamma>3/2$,  we have $\rho_1(\gamma) = \rho_2(\gamma) = \gamma>3/2$. Let $w_1=f'(u(t))$ and apply Lemma \ref{le:multi} with $(\varepsilon,d/2-\varepsilon,\varepsilon)\in \mathcal{B}$,   combined with \eqref{SD_G_2}, to derive
$$ 
\Vert f'(u^h(t))-f'(u(t))\Vert_\varepsilon\leq C \Vert A^{(d/2-\varepsilon)/2}f''(u^h(t)+u(t))\Vert \Vert u(t)-u^h(t)\Vert_{\varepsilon} \leq C h^{\gamma-\varepsilon}.
$$
This, together with the boundedness of
$\Vert A^{\rho_1(\gamma)/2}f'(u(t))\Vert$ and Lemma \ref{le:mult_algebra}, yields \eqref{frech_first}. It implies that 
\begin{equation}\label{bound_g}
	\Vert A_h^{s/2} f_h(u^h(t))\Vert\leq \Vert A_h^{s/2} P_h(f'(u^h(t))u^h(t))\Vert+C\leq C(t^{\gamma/2-\sigma(s)/2}+1).
\end{equation}

For the estimation of the second- and third-order Fréchet derivatives, involving the parameters $\rho_2(\gamma)$ and $\rho_3(\gamma)$, the derivation is quite tedious. To simplify the presentation, we extract some common features and make the following assumptions.

\begin{assumption}\label{ass:Frech_dominant}
	For $-2\leq s<\rho_1(\gamma)$, the following estimates hold
	\begin{align}
		\Vert A_h^{s/2}P_h (f^{(2)}(u^h(t))v_1^h v_2^h) \Vert &\leq C  \Vert A_h^{s_{21}/2}v^h_1 \Vert \Vert A_h^{s_{22}/2}v^h_2 \Vert, \label{frech_second}
		\\
		\Vert A_h^{s/2}P_h (f^{(3)}(u^h(t))v^h_1 v^h_2 v^h_3) \Vert &\leq C  \Vert A_h^{s_{31}/2}v^h_1 \Vert \Vert A_h^{s_{32}/2}v^h_2 \Vert \Vert A_h^{s_{33}/2}v^h_3 \Vert, \label{frech_third}
	\end{align}
	where $s_{21}$, $s_{22}$, $s_{31}$, $s_{32}$, $s_{33}\geq s$, $s_{21}+s_{22}\leq \gamma + \sigma(s)$,  $s_{31}+s_{32}+s_{33}\leq 2\gamma + \sigma(s)$.
\end{assumption}
Dominance of the first-order derivative is guaranteed by this assumption. Specifically, when estimate \eqref{AhDtuh_1} in Proposition \ref{pro:AhDtuh_1} holds, we can consequently derive the validity of \eqref{AhDtg}. In what follows, illustrative examples are provided to validate the assumption.

\begin{example}\label{example:d2case}
	$f(\tau)$ satisfies  Assumptions \ref{assum:nf_smooth}-\ref{assum:bound_compos} and $\frac{d}{2}<\gamma\leq 2$. According to the Assumption \ref{assum:bound_compos} and the bounds \eqref{bound_u}, \eqref{bound_uh}, when $\gamma>\frac{d}{2}$, we have
	$$
	\Vert A^{\hat{\gamma}/2} f^{(m)}(u^h(t))\Vert + \Vert A^{\gamma/2} f^{(m)}(u(t))\Vert\leq C.$$ 
	Observing the set $\mathcal{B}$, when $s_1$ approaches or exceeds $\frac{d}{2}$, in order to make $s_2$ as small as possible, $s_2$ is a slight shift of $s$, with the constraint that $s_1 + s_2 > 0$. This case can be generalized. When $\frac{d}{2}<\gamma<\frac{3}{2}$, by progressively employing Lemma \ref{le:multi} with
	$(s,\gamma,\sigma(s))$, $(\sigma(s),\gamma,\sigma(s))$,...,$(\sigma(s),\gamma,\sigma(s))\in \mathcal{B}$, and applying \eqref{Auh} and \eqref{AhPh}, we get the following estimate for derivatives \eqref{Dfh_def}.
	\begin{align}\label{mult_d2case}
		\Vert A_h^{s/2}P_h(f^{(m)}(u^h(t)) v^h_1...v^h_m)\Vert\leq C\Vert A_h^{\gamma/2}v^h_1\Vert \Vert A_h^{\gamma/2}v^h_2\Vert...\Vert A_h^{\sigma(s)/2}v^h_m\Vert,
	\end{align}
	where $v^h_i\in V_h$, $i=1,2,...,m$ and $-2<s<\gamma$. This implies 
	$$s_{21}+s_{22} = \gamma + \sigma(s,\gamma), \quad s_{31}+s_{32} + s_{33} = 2\gamma + \sigma(s,\gamma).$$ 
	When $3/2<\gamma\leq 2$, the same conclusion holds by applying estimates \eqref{mult_Xh_m1} and \eqref{mult_Xh_m2}.
\end{example}

\begin{example}  \label{example:less_d2case}
	Consider the case where $\beta_1<\gamma<\frac{d}{2}$, and follow the analysis in Examples \ref{example:poly} and \ref{example:ration}. According to Assumption \ref{assum:bound_compos}, we only know $\Vert f^{(m)}(u^h(t)) \Vert_{\rho_m(\gamma)}\leq C$ for the nonlinear term. To estimate the Fréchet derivative \eqref{Dfh_def}, we define the following set, as established in \cite[Theorem 4.5.1]{SobFracNEMy}.
	\begin{equation}\label{Bm}
		\begin{aligned}
			\mathcal{B}_{m}= \left\{(s,s_1,s_2,...,s_m) \:\big|\: 0<s_i<\frac{d}{2},\: -\frac{d}{2}\leq s<s_i,\: i=1,2,...,m,\: \frac{d}{2}-s>\sum_{j=1}^{m}(\frac{d}{2}-s_j) \right\}.
		\end{aligned}
	\end{equation}
	Then for any $v^h_i\in X_h$, $i=1,2,...,m$ with $(s,\rho_m(\gamma),s_{m1},s_{m2},...,s_{mm})\in \mathcal{B}_{m+1}$, we have 
	\begin{align}\label{multi_Xh_m}
		\Vert A_h^{s/2}P_h(f^{(m)}(u^h(t))\prod_{i=1}^m v^h_i)\Vert\leq C\Vert f^{(m)}(u^h(t))\Vert_{\rho_m(\gamma)} \Vert A_h^{s_{m1}/2}v^h_1\Vert...\Vert A_h^{s_{mm}/2}v^h_m\Vert.
	\end{align}
	The proof of the above estimate is similar to \cite[Lemma 1]{ERK_FNM}. When $f(\tau)=\tau^m$, $m\geq 2$, we have $\rho_k(\gamma)=\gamma-(m-k-1)(\frac{d}{2}-\gamma)$. We derive the estimates \eqref{frech_second} and \eqref{frech_third} by applying \eqref{multi_Xh_m} with $(s,\rho_2(\gamma), s_{21}, s_{22})\in \mathcal{B}_3$ and $(s,\rho_3(\gamma), s_{31}, s_{32}, s_{33})\in \mathcal{B}_4$, respectively. It can be verified that both conditions $s_{21}+s_{22}\leq \gamma + \sigma(s)$,  $s_{31}+s_{32}+s_{33}\leq 2\gamma + \sigma(s)$ hold. When $f(\tau)=\tau^4/(1+\tau^2)$ and $d=2$, we obtain $\rho_1(\gamma)=\rho_2(\gamma)=\gamma$. 
	After using estimate \eqref{multi_Xh_m}, the constraint $s_{21}+s_{22}\leq \gamma + \sigma(s)$ does not hold. The polynomial case exactly satisfies the required conditions due to a $d/2 - \gamma$ differential dimensional decrease in nonlinear terms after differentiation. Note that $\Vert f''(u^h(t))\Vert_\infty$ is uniform bound. Its differential dimension ($0-d/\infty$) is associated with $H^{d/2}$ ($d/2-d/2$). This behavior is analogous to that observed in polynomial cases. We now proceed with this perspective. Based on \cite[Chapter 4.5]{SobFracNEMy}, it follows that 
	$$ 
	\Vert f^{(2)}(u^h(t))v_1^h v_2^h \Vert \leq C  \Vert f''(u^h(t))\Vert_\infty \Vert A_h^{s'/2}v^h_1 \Vert \Vert A_h^{-s/2}v^h_2 \Vert,   $$
	where $s'=1+s+\varepsilon$, $s<0$, $s'>0$. By duality argument, we get
	$$ 
	\Vert A_h^{{s}/2}P_h (f^{(2)}(u^h(t))v_1^h v_2^h) \Vert \leq C  \Vert f''(u^h(t))\Vert_\infty \Vert A_h^{s'/2}v^h_1 \Vert \Vert v^h_2 \Vert,   $$
	for $-2\leq s< 0$. This implies $s_{21}+s_{22}=s'+0\leq \gamma+\sigma(s)$. 
\end{example}

\subsection{Estimation for $\mathcal{K}_m^{\hat{p}}$}
Having validated the Assumption \ref{ass:Frech_dominant} through some examples, we now turn to estimate the equivalent terms in $\mathcal{K}_m^{\hat{p}}$. We begin by analyzing the derivatives of $u^h(t)$. 
\begin{proposition}\label{pro:AhDtuh_1}
	Let Assumptions \ref{assum:nf_smooth}-\ref{ass:Frech_dominant} be fulfilled. If $u^h(t)$ is the solution of \eqref{eq2}, then $u^h(t)\in C^m((0,T]; X_h)$ and the following estimates hold for $m=0,1,2,3$:
	\begin{align}
		\Vert A_h^{s/2} D_t^m u^h(t)\Vert&\leq Ct^{-m+{\gamma/2}-s/2}+C,  &&-2\leq s\leq 2, \label{AhDtuh_1}&
		\\
		\Vert A_h^{s/2} D_t^m g(t)\Vert&\leq Ct^{-m+{\gamma/2}-\sigma(s)/2}+C, && -2\leq s < \rho_1(\gamma).\label{AhDtg}&
	\end{align}
	
\end{proposition}

\begin{proof}
	Based on equation \eqref{eq2} and Assumption \ref{assum:nf_smooth}, $u^h(t)$ is differentiable of any order on $(0,T]$. 
	From \eqref{bound_uh} and \eqref{bound_g}, the formulas \eqref{AhDtuh_1} and \eqref{AhDtg} holds for $m=0$. 
	Assume \eqref{AhDtuh_1} holds for $m\leq l-1$ ($l=1,2,3$). In the following, we prove that it also holds for $m = l$. 
	
	We begin by making some preliminary preparations. 
	Using the chain rule and Assumption \ref{ass:Frech_dominant}, it holds that 
	\begin{equation}
		\begin{aligned}
			\Vert A_h^{s/2} D_t^m f_h(u^h(t)) \Vert &\leq  C\Vert A_h^{s/2}P_h( f'(u^h(t)) D_t^m u^h(t)) \Vert 
			\\
			& \quad + C \sum_{i=2}^{m}\:\:  \sum_{\substack{\alpha_1+..+\alpha_i=m\\\alpha_i\in N^+}}\Vert A_h^{s/2} P_h (f^{(i)}(u^h(t)) D_t^{\alpha_1} u^h(t)...D_t^{\alpha_i} u^h(t)) \Vert\\
			& \leq Ct^{-m+\gamma/2-\sigma(s)/2}+C,
		\end{aligned}
	\end{equation}
	where $0\leq m\leq l-1, -2\leq s< \rho_1(\gamma)$. If \eqref{AhDtuh_1} is valid for $m=l$, so does \eqref{AhDtg}. Additionally,  we set $V(t)=t^l D_t^l u^h(t)$. As in \cite[Lemma 3.7]{mukam}, it holds that
	\begin{align*}
		D_t V(t) &= lt^{l-1} D_t^l u^h(t)+ t^l D_t^{l+1} u^h(t) \\
		&= lt^{l-1} D_t^l u^h(t) + t^l \left(A_h D_t^l u^h(t)+D_t^l f_h(u^h(t))\right)\\
		&= A_h V(t) + lt^{l-1} D_t^l u^h(t) + t^l D_t^l f_h(u^h(t)).
	\end{align*}  
	Therefore by variation-of-constants formula, we have
	\begin{equation}\label{Du_voc}
		D_t^l u^h(t)=t^{-l}\int_{0}^{t}S_h(t-\tau)(l\tau^{l-1}D_t^l u^h(\tau) +\tau^l D_\tau^l f_h(u^h(\tau)))\text{d}\tau.
	\end{equation}   
	
	Now turn to the proof of estimate \eqref{AhDtuh_1}. We first consider the case $-2\leq s\leq 0$. It can be verified that
	\begin{align*}
		\Vert A_h^{s/2} D^{l}_t u^h(t)\Vert&\leq  \Vert A_h^{1+s/2}D_t^{l-1} u^h(t)\Vert +\Vert A_h^{s/2} D_t^{l-1}f_h(u^h(t))\Vert\\
		&\leq C t^{-l+\gamma/2-s/2}+  C t^{-l+1+\gamma/2-\sigma(s)/2}+C\\
		&\leq C t^{-l+\gamma/2-s/2}+C.
	\end{align*}
	From the above bound, we know the case of $s=0$, $m=l$. Select sufficiently small $\varepsilon'$ such that $(-\frac{d}{2}+2\varepsilon',\rho_1(\gamma),0)\in\mathcal{B}$. For $0<s<2-\frac{d}{2}+2\varepsilon'$, by  formula \eqref{Du_voc} and Assumption \ref{ass:Frech_dominant}, we deduce that
	\begin{equation}\label{Du1_res}
		\begin{aligned}
			\quad\:\: \|A_h^{s/2} D_t^l u^h(t)\|&\leq t^{-l}\int_0^t  \Vert A_h^{s/2}S_h(t-\tau)l\tau^{l-1} D_\tau^l u^h(\tau)\Vert \text{d}\tau
			\\
			&\qquad +t^{-l}\int_0^t \Vert A_h^{s/2} S_h(t-\tau) A_h^{d/4-\varepsilon'}\tau^l A_h^{-d/4+\varepsilon'} D^l_\tau f_h(u^h(\tau))\Vert \text{d}\tau
			\\
			&\leq C t^{-l}\int_0^t  (t-\tau)^{-s/2} \tau^{-1+\gamma/2} \text{d}\tau + C t^{-l}\int_0^t (t-\tau)^{-s/2-d/4+\varepsilon'} \tau^{\gamma/2} \text{d}\tau
			\\
			&\leq  Ct^{-l+\gamma/2-s/2}+ C.
		\end{aligned}
	\end{equation}
	Following the same approach as above, we can continue to expand the range of values for $s$. Observe the right-hand side of the inequality. For the first term, if $s=2$, the integrand is non-integrable. To address this, we let $D^{l}_\tau u^h(\tau)$ absorb an operator $A_h^{\varepsilon'}$. For the second term, the extension is made incrementally by successively applying   estimate \eqref{mult_Xh} with $(\min(\varepsilon',2+\varepsilon'-d),\rho_1(\gamma),2-\frac{d}{2}-\varepsilon')$, $(\varepsilon',\rho_1(\gamma),\frac{d}{2}-\varepsilon')\in\mathcal{B}$. This completes the proof.
\end{proof}

\begin{proposition}\label{pro:Km1}
	Let Assumptions \ref{assum:nf_smooth}-\ref{ass:Frech_dominant} be fulfilled. For any $v^h(t)\in \mathcal{K}_{m}^{\hat{p}} \setminus  \mathcal{K}'$, the following estimates hold.
	\begin{enumerate}[(i).]
		\item When $\hat{p}=2:$
		\begin{itemize}
			\item $m=2:$ $\|A_h^{s/2}v^h(t)\| \leq C(t^{-1+\gamma/2-\sigma(s)/2} + 1)$ for $-2\leq s<\rho_1(\gamma)$.
			\item $m\geq 3:$ $\delta^m\|A_h^{s/2}v^h(t)\| \leq C(\delta^3 t^{-2+\gamma/2-\sigma(s)/2} + 1)$ for $s\geq -2$ with $\sigma(s)<\rho_1(\gamma)$.
		\end{itemize}
		\item When $\hat{p}=3:$
		\begin{itemize}
			\item $m=2,\:3:$ $\|A_h^{s/2}v^h(t)\| \leq C(t^{-m+1+\gamma/2-\sigma(s)/2} + 1)$ for $s\geq -2 $ with $\sigma(s)<\rho_1(\gamma)$.
			\item $m\geq 4:$ $\delta^m\|A_h^{s/2}v^h(t)\| \leq C(\delta^4 t^{-3+\gamma/2-\sigma(s)/2} + 1)$ for $s\geq -2$ with $\sigma(\sigma(s))<\rho_1(\gamma)$.
		\end{itemize}
	\end{enumerate}
\end{proposition}
\begin{proof}
	We restrict our proof to the case $\hat{p}=3$, and the proof for $\hat{p}=2$ follows through analogous arguments. The functions in $\mathcal{K}^3_m\setminus \mathcal{K}'$ can be expressed using compact symbolic notation, with the cases $m=2,3,4$ given as follows: 
	\begin{itemize}
		\item $m=2$ : $f'u'$.
		\item $m=3$ : $f'u''$, $f''u'u'$, $f'f'u'$.
		\item $m=4$ : $f'u'''$, $f'f'u''$, $f'f''u'u'$, $f'f'f'u'$,  $f''u'u''$, $f''u'f'u'$, $f'''u'u'u'$.
	\end{itemize}
	By employing estimate \eqref{frech_first} together with Assumption \ref{ass:Frech_dominant} and the constraint on $s$ specified in the proposition, we can rigorously verify that the estimates for the functions listed above hold precisely as stated in the proposition. Subsequently, we assume the proposition's estimates for $\hat{p}=3$, $m>4$ hold when $m < l-1$. It remains to prove the case $m = l$.  According to the recursive definition, there exist three cases. We shall examine the first case $\delta^l\Vert A_h^{s/2}S_h(\delta)P_h(f'(u^h(t))v^h(t))\Vert\leq  C(\delta^4 t^{-3+\gamma/2-\sigma(s)/2} + 1)$ as an illustrative example, where $v^h(t)\in \mathcal{K}^{3}_{l-1}$. There are two cases for $v^h(t)$. The first case $v^h(t)\in \mathcal{K}'$, where the values of $s$ are not strictly constrained, and the estimate holds when substituted. The second case $v^h(t)\in \mathcal{K}^{3}_{l-1}\setminus \mathcal{K}'$, in which case $v^h(t)$ is multiplied by an exponential function $S_h(\delta)$, and the following technique can be used for estimation: 
	\begin{align*}
		\delta^l\Vert A_h^{s/2}S_h(\delta)P_h(f'(u^h(t))v^h(t)) \Vert &\leq  \delta^{l}\Vert A_h^{\sigma(s)/2-s/2} A_h^{s/2}v^h(t)\Vert 
		\\
		&\leq C\delta^{1-s/2+\sigma(s)/2}(\delta^4 t^{-3+\gamma/2-\sigma(s)/2} + 1).
	\end{align*}
	This completes the proof.
\end{proof}  

\subsection{Main result}
Now, the main result of this paper is formulated in the following theorem.
\begin{theorem}\label{theo:tempErr1}
	Suppose the initial value problem \eqref{eq1} satisfies Assumptions \ref{assum:nf_smooth}--\ref{ass:Frech_dominant}. Consider for its numerical solution a $p$th-order explicit exponential Runge-Kutta method \eqref{NumSche1} with a constant stepsize $\delta$. For $p\geq 2$, if there exist $r\in(0,2)$ such that $\sigma(-r)<\rho_1(\gamma)$ and $r'\in(0,\rho_1(\gamma))$, we obtain the temporal error for $\hat{p}=2$:
	\begin{align}\label{res_theo3}
		\Vert e_{n+1}\Vert\leq C t_n^{-r/2} \delta^{\min(1+\gamma/2-\sigma(-r)/2,\:\hat{p})}
		+  C t_{n}^{r'/2-1} \delta^{\min(2+{\gamma}/{2}-{\sigma(r')}/{2},\:\hat{p})},
	\end{align}
	where $C$ is independent of $h$, $\delta$, and $n$, and the function $\sigma(\cdot)$ is defined by \eqref{eq:sigma} in Section 5.1. Furthermore, for $p\geq 3$, if there exist $r\in(0,2)$ and $r'\in (0,\rho_1(\gamma))$ such that $\sigma(\sigma(-r))<\rho_1(\gamma)$ and $\sigma(r')<\rho_1(\gamma)$, then set $\hat{p}=3$ in \eqref{res_theo3}.
\end{theorem}

\begin{remark}
	The parameter $p$ corresponds to the order of the local error expansion. Higher-order methods yield both second- and third-order expansions, allowing selection of the optimal result. Although the third-order expansion requires stronger conditions on $r$ and $r'$, these constraints relax as $\gamma$ increases. If we are only concerned with the convergence order of the method, choosing $r = 2-\varepsilon$ and $r'= \varepsilon$ results in a convergence order of $1+{\gamma}/{2}+\rho_1(\gamma)/2$. The $L^2$-error analysis presented in the above theorem can be naturally extended to $H^1$ estimates by following the approach in \cite{ERK_FNM}.
\end{remark}

\begin{proof}
	According to \eqref{err_L2Spli}, there exist $r_1\in (0,\frac{d}{2})$ such that
	\begin{align}\label{mid_theore2_en}
		\Vert e_{n+1}\Vert \leq \left\Vert \sum_{k=0}^{n} S_h(t_{n-k})\tilde{e}_{k+1}\right\Vert + \delta \sum_{k=1}^{n} t_{n+1-k}^{-r_1/2} \Vert e_k \Vert.
	\end{align}
	It remains to estimate the first term of the right-hand side. 
	
	Consider the case where $0<k<n$. The local error $\tilde{e}_{k+1}$ has the expansion form
	\begin{align*}
		\tilde{e}_{k+1}
		=\tilde{e}_{k+1}^{(\hat{p})}+\mathbf{O}_{\hat{p}+1}(t_k),
	\end{align*}
	where the descriptions of $\tilde{e}_{k+1}^{(\hat{p})}$ and $\mathbf{O}_{\hat{p}+1}(t_k)$ given in \eqref{Err_localReprens}. As discussed at the end of Section \ref{sec:NumSche_ErrExpa}, it can be equivalently replaced by $V_{\hat{p}} (t_k)$ and $\delta^{\hat{p}+1} V_{\hat{p}+1}(t_k)+\delta^{\hat{p}+2}V_{\hat{p}+2}(t_k)+...$. Under the theorem's conditions on $r$ and by Proposition \ref{pro:Km1}, it holds that
	\begin{align*}
		\Vert S_h(t_{n-k})\textbf{O}_{\hat{p}+1}(t_{k})\Vert
		\leq &\: C \Vert S_h(t_{n-k})A_h^{r/2}\Vert_{\mathcal{L}(X_h)} \: \delta^{\hat{p}+1} t_k^{-\hat{p}+\gamma/2-\sigma(-r)/2} \\
		\leq &\: C \delta^{\hat{p}+1}  t_{n-k}^{-r/2} t_{k}^{-\hat{p}+\gamma/2-\sigma(-r)/2}.
	\end{align*} 
	There exist bounded operators $\tilde{\psi}_p$ and $\tilde{b}_i$ with
	\begin{align*}
		\psi_p(\delta A_h)-\psi_p(0)=\delta A_h \tilde{\psi}_p,\quad b_i(\delta A_h)-b_i(0) = \delta A_h \tilde{b}_i.
	\end{align*}
	From the above formulas and the theorem's conditions on $r'$, we get
	\begin{align*}
		\Vert S_h(t_{n-k}) \tilde{e}_{k+1}^{(\hat{p})}\Vert
		\leq &\:\delta^{\hat{p}}\Vert S_h(t_{n-k})A_h \delta A_h^{-r'/2} A_h^{r'/2} V_{\hat{p}}(t_{k})\Vert\\
		\leq & \: \delta^{\hat{p}+1} t_{n-k}^{r'/2-1} t_{k}^{-\hat{p}+1+\gamma/2-\sigma(r')/2}.
	\end{align*}
	
	Consider the case where $k=n$. 
	$$
	\Vert \tilde{e}_{n+1}\Vert\leq C\delta^{\hat{p}} \Vert V_{\hat{p}}(t_n)\Vert \leq C\delta^{\hat{p}} t_n^{-\hat{p}+1+\gamma/2-\sigma(0)/2}.$$
	
	Consider the case where $k=0$. The error expansion is located at the initial point, but some estimates of $D_t^m u^h(0)$ is unknown. There are two approaches to analyze the error of the EERK method. One is from \cite{luan12stiff,luan14stiff}, which was used in above analysis. However, according to \eqref{recur1_end}, the terms in the expansion of local error require an estimate of $D_t^m u^h(0)$. The other approach, introduced by \cite{ERK1},  uses $D_t^m u^h(\theta \delta_0)$ rather than $D_t^mu^h(0)$. Noting that $\tilde{e}_1=e_1$, we use this approach to estimate $\Vert S_h(t_n)\tilde{e}_1 \Vert$. From the Lipschitz condition \eqref{Lips2}, we get
	\begin{equation}\label{Ae1_mid1}
		\begin{aligned}
			\Vert A_h^{-r/2}\tilde{e}_1\Vert &=\left\Vert A_h^{-r/2}\delta \int_{0}^{1} S_h((1-\theta)\delta)\sum_{i=1}^{\kappa}l_i(\theta)\left(f_h(\widehat{U}^h_{0i})-f_h(u^h(\theta \delta))\right)\text{d}\theta\right\Vert
			\\
			&\leq  \delta \sum_{i=1}^{\kappa}\Vert A_h^{\sigma(-r)/2} \widetilde{E}_{0i}\Vert
			+\delta \left\Vert \int_{0}^{1} A_h^{-r/2} S_h((1-\theta)\delta)\sum_{i=1}^{\kappa}l_i(\theta) \big(g({c_i\delta})-g(\theta\delta)\big)\text{d}\theta\right\Vert,
		\end{aligned}
	\end{equation}
	where $\widetilde{E}_{0i}=\widehat{U}_{0i}^h-\widetilde{U}_{0i}^h$, 
	$b_i(\delta A_h)= \int_0^1 S_h((1-\theta)\delta)l_i(\theta)\text{d}\theta$ and $\sum_{i=1}^{\kappa}l_i(\theta)=1$ by \eqref{orderCond_1}. To prove $\Vert A_h^{-r/2}\tilde{e}_1\Vert\leq C \delta^{\min(1+\gamma/2-\sigma(-r)/2,\:\hat{p})}$, we first analyze the second term in \eqref{Ae1_mid1}. When $\gamma-\sigma(-r)\leq 2$, expanding $g(t_{0i})$ and $g(\theta \delta)$ in Taylor series at $t_0$ up to the first derivative, employing the estimate \eqref{AhDtg}, we obtain
	\begin{align*}
		&\left\Vert \int_{0}^{1} A_h^{-r/2} S_h((1-\theta)\delta)\sum_{i=1}^{\kappa}l_i(\theta)\left( g(t_{0i})-g(\theta \delta)\right)\right\Vert\\
		\leq &\: \left\Vert \int_{0}^{1} A_h^{-r/2} S_h((1-\theta)\delta)\delta \sum_{i=2}^{\kappa}l_i(\theta) \left( \int_{0}^{1}g'(\xi c_i\delta)\text{d}\xi - \int_{0}^{1}g'(\xi\theta\delta)\text{d}\xi\right)\right\Vert
		\\
		\leq &\: C\delta\sum_{i=2}^{\kappa} \left(\int_{0}^{1}(\xi c_i\delta)^{-1+\gamma/2-\sigma(-r)/2}\text{d}\xi+ \int_{0}^{1}\int_{0}^{1}(\xi\theta\delta)^{-1+\gamma/2-\sigma(-r)/2}\text{d}\xi \text{d}\theta \right)\\
		\leq &\: C\delta^{\gamma/2-\sigma(-r)/2}.
	\end{align*}
	When $2<\gamma-\sigma(-r)\leq 4$ and the method fulfill the second-order condition $\sum_{i=2}^{\kappa}b_ic_i=\varphi_2$, expanding $g(t_{0i})$ and $g(\theta \delta)$ in Taylor series at $t_0$ up to the second derivative and applying the estimate \eqref{AhDtg} yields
	\begin{align*}
		&\left\Vert \int_{0}^{1} A_h^{-r/2} S_h((1-\theta)\delta) \sum_{i=1}^{\kappa}l_i(\theta)\left( g(t_{0i})-g(\theta \delta)\right)\right\Vert\\
		\leq &\: \left\Vert \int_{0}^{1} A_h^{-r/2} S_h((1-\theta)\delta)\delta^2 \sum_{i=2}^{\kappa}l_i(\theta) \left( \int_{0}^{1}(1-\xi)g''(\xi c_i\delta)\text{d}\xi + \int_{0}^{1}(1-\xi)g''(\xi\theta\delta)\text{d}\xi\right)\right\Vert\\
		\leq &\: C\delta^{2}\left(\sum_{i=2}^{\kappa} \int_{0}^{1}(1-\xi)(\xi c_i\delta)^{-2+\gamma/2-\sigma(-r)/2}\text{d}\xi+ \int_{0}^{1}\int_{0}^{1}(1-\xi)(\xi\theta\delta)^{-2+\gamma/2-\sigma(-r)/2}\text{d}\xi \text{d}\theta \right)\\
		\leq &\: C\delta^{\gamma/2-\sigma(-r)/2}.
	\end{align*}
	
	The next step is to prove $\Vert A_h^{\sigma(-r)/2}\widetilde{E}_{0i}\Vert \leq C\delta^{\gamma/2-\sigma(-r)/2}$. Similar to \eqref{Ae1_mid1}, we obtain
	\begin{align*}
		\Vert A_h^{\sigma(-r)/2}\widetilde{E}_{0i}\Vert &=\left\Vert \delta \int_{0}^{1} A_h^{\sigma(-r)/2+r/2}S_h((1-\theta)c_i\delta)\left(\sum_{j=1}^{i-1}l_{i,j}(\theta)A_h^{-r/2}(f_h(\widehat{U}_h^{0,j})-f_h(U(\theta \delta)))\right)\right\Vert\\
		&\leq C\delta^{1-\sigma(-r)/2-r/2} \left(\sum_{j=1}^{i-1}\Vert A_h^{\sigma(-r)/2} \widetilde{E}_{0,j}\Vert+\delta^{\gamma/2-\sigma(-r)/2}\right),
	\end{align*}
	where $a_{ij}(\delta A_h)= \int_0^1 S_h((1-\theta)c_i\delta)l_{ij}(\theta)\text{d}\theta$ and $\sum_{j=1}^{i-1}l_{ij}(\theta)=c_i$ by \eqref{orderCond_1}. Since $i$ is finite, the conclusion can be drawn by recursion. Here we have verified the first-step local error only for orders up to three. Attaining higher orders would require a higher-order Taylor expansion of $f_h$ in \eqref{Ae1_mid1}, and the imposition of corresponding higher-order conditions (see \cite{ERK1} for details). In conclusion, using \cite[Lemma 6.1]{IRK2} for summation, we have
	\begin{equation*}
		\begin{aligned}
			\left\Vert \sum_{k=0}^{n} S_h(t_{n-k})\tilde{e}_{k+1}\right\Vert & \leq  \sum_{k=1}^{n-1} C \delta^{\hat{p}+1} t_{n-k}^{-r/2} t_{k}^{-\hat{p}+\gamma/2-\sigma(-r)/2} 
			+ \sum_{k=1}^{n-1} C \delta^{\hat{p}+1} t_{n-k}^{r'/2-1} t_{k}^{-\hat{p}+\gamma/2+1-\sigma(r')/2}
			\\
			&\quad + C\delta^{\hat{p}} t_n^{-\hat{p}+1+\gamma/2-\sigma(0)/2} + t_n^{-r/2} \delta^{\min(1+\gamma/2-\sigma(-r)/2,\:\hat{p})}				
			\\
			&\leq C t_n^{-r/2} \delta^{\min(1+\gamma/2-\sigma(-r)/2,\:\hat{p})}
			+  C t_{n}^{r'/2-1} \delta^{\min(2+\gamma/2-\sigma(r')/2,\: \hat{p})}.
		\end{aligned}
	\end{equation*}
	Substituting into \eqref{mid_theore2_en} and using Gronwall's inequality(see \cite[Lemma 6.2]{IRK2}) completes the proof.
\end{proof}

\section{Numerical Experiments}\label{sec:NumExpe}
In this section, we present numerical tests to support the theoretical analysis in Section \ref{sec:ErrEst}. We consider the semilinear parabolic problem \eqref{eq1} in the domain $\Omega = (0,1)\times (0,1)$ up to $T = 1$. The sectorial operator $A$ is realization of $\Delta-I$ in $L^2$ under the homogeneous Neumann boundary condition on $\partial \Omega$. The spatial discretization is performed using the linear Galerkin finite element method \eqref{eq2}  with an adequately small spatial mesh $h=1/64$, while the temporal discretization relies on the third-order EERK method \eqref{EERK3}. 

The main focus is to compute the convergence order of the third-order EERK method at $T$ for semilinear parabolic problems \eqref{eq2} with various initial values and nonlinear terms. We consider the following initial values.
\begin{enumerate}[(i).]
	\item $u_0(x_1,x_2)=0.5\:\text{sign}(x_2-0.5)+1.3\in D(A^{1/4-\varepsilon})\cap L^\infty$.
	\item $u_0(x_1,x_2)=(x_1^2+x_2^2)^{-1/4}\in D(A^{1/4-\varepsilon})$.
	\item $u_0(x_1,x_2)=0.5(x_1^2+x_2^2)+1\in D(A^{3/4-\varepsilon})$.
	\item  $u_0(x_1,x_2) = (2x_1^{3/2}-x_1^3)(2x_2^{3/2}-x_2^3) +1 \in D(A^{1-\varepsilon})$.
\end{enumerate}
Substitute these initial values into the semilinear parabolic problem with the following nonlinear terms.
\begin{enumerate}[(1).]
	\item $f(u)=-(u+1)(u-1.5)+u$ with $\rho_1(\gamma)=\gamma$. 
	\item $f(u)= -1/8\:(u+1)^2(u-1.5)+u$ with  $\rho_1(\gamma)=2\gamma-1$. 
	\item $f(u)=-u^4/(1+u^2)+u+81/52$ with  $\rho_1(\gamma)=\gamma$.
\end{enumerate}
It can be verified through Examples \ref{example:poly}-\ref{example:less_d2case} that Assumptions \ref{assum:nf_smooth}-\ref{ass:Frech_dominant} hold. The numerical orders of convergence are computed by
\begin{equation*}
	\log\left(\frac{\Vert U_{T,N}-U_{T,2N} \Vert } {\Vert U_{T,2N}-U_{T,4N} \Vert}\right)/\log(2),
\end{equation*} 
where $U_{T,N}$ represents the value at $T$, obtained by applying the third-order EERK method to the discretized problem \eqref{eq2} with a constant step size of $T/N$. For $\gamma<1$, Theorem \ref{theo:tempErr1} yields a convergence order of $1+{\gamma}/{2}+\rho_1(\gamma)/{2}-\varepsilon$. The numerical results for initial values (\rmnum{1}) and (\rmnum{2}) in Tables \ref{tab:fun1}-\ref{tab:fun3} align with the theoretical predictions. Notably, initial value (\rmnum{1}) belongs to $L^{\infty}$, while for the nonlinear term in case $(2)$, the solution $\Vert u^h(t)\Vert_{\infty}$ remains uniformly bounded with respect to the $h$. As shown in Example \ref{example:poly}, we have $\rho_1(1/2)=1/2$ in this case. For $\gamma>1$, Theorem \ref{theo:tempErr1} yields a convergence order of $1+\gamma-\varepsilon$. Numerical results for initial values (\rmnum{3}) and (\rmnum{4}) in Tables \ref{tab:fun1}-\ref{tab:fun3} validate this theoretical analysis. 

\begin{table}[tbhp]
	\centering
	\caption{The temporal discretization  convergence orders for model with nonlinear term ($1$) and initial values (\rmnum{1})-(\rmnum{4}). Theoretical convergence orders are listed in the last row.}
    \label{tab:fun1}
	\begin{tabular}{ccccccccccccc}
		\hline
		\multirow{2}{*}{$N$} & & \multicolumn{2}{c}{initial data (\rmnum{1})} &&  \multicolumn{2}{c}{initial data (\rmnum{2})}  && \multicolumn{2}{c}{initial data (\rmnum{3}) } && \multicolumn{2}{c}{initial data (\rmnum{4})}
		\\ \cline{3-4} \cline{6-7} \cline{9-10} \cline{12-13} 
		~& & Error & Order &&  Error & Order &&  Error & Order &&  Error & Order 
		\\ \hline
		$2^6$ & ~ & 5.180E-06 & -- & ~ & 5.873E-06 & -- & ~ & 1.130E-07 & -- & ~ & 1.057E-07 & -- \\ 
        $2^7$ & ~ & 1.758E-06 & 1.559 & ~ & 1.967E-06 & 1.578 & ~ & 2.020E-08 & 2.484 & ~ & 1.376E-08 & 2.941 \\ 
        $2^8$ & ~ & 5.935E-07 & 1.566 & ~ & 6.612E-07 & 1.572 & ~ & 3.598E-09 & 2.489 & ~ & 1.787E-09 & 2.945 \\ 
        $2^9$ & ~ & 1.978E-07 & 1.586 & ~ & 2.215E-07 & 1.578 & ~ & 6.417E-10 & 2.487 & ~ & 2.322E-10 & 2.944 \\  
        --&&-- & $1.5-\varepsilon$ &&--& $1.5-\varepsilon$ && --& $2.5-\varepsilon$ &&--&$3-\varepsilon$ \\
        \hline
	\end{tabular}
\end{table}

\begin{table}[tbhp]
	\centering
	\caption{The temporal discretization   convergence orders for model with nonlinear term ($2$) and initial values (\rmnum{1})-(\rmnum{4}). Theoretical convergence orders are listed in the last row.}
    \label{tab:fun2}
	\begin{tabular}{ccccccccccccc}
		\hline
		\multirow{2}{*}{$N$} & & \multicolumn{2}{c}{initial data (\rmnum{1})} &&  \multicolumn{2}{c}{initial data (\rmnum{2})}  && \multicolumn{2}{c}{initial data (\rmnum{3}) } && \multicolumn{2}{c}{initial data (\rmnum{4})}
		\\ \cline{3-4} \cline{6-7} \cline{9-10} \cline{12-13} 
		~& & Error & Order &&  Error & Order &&  Error & Order &&  Error & Order 
		\\ \hline
		$2^6$ & ~ & 1.567E-05 & -- & ~ & 6.396E-05 & -- & ~ & 3.907E-07 & -- & ~ & 2.770E-07 & -- \\ 
        $2^7$ & ~ & 5.394E-06 & 1.538  & ~ & 2.555E-05 & 1.324  & ~ & 7.232E-08 & 2.434  & ~ & 3.594E-08 & 2.946   \\ 
        $2^8$ & ~ & 1.834E-06 & 1.556  & ~ & 1.014E-05 & 1.333  & ~ & 1.322E-08 & 2.452  & ~ & 4.654E-09 & 2.949   \\ 
        $2^9$ & ~ & 6.135E-07 & 1.580  & ~ & 3.991E-06 & 1.345  & ~ & 2.403E-09 & 2.460  & ~ & 6.032E-10 & 2.948   \\ 
        --&&-- & $1.5-\varepsilon$ &&--& $1.25-\varepsilon$ && --& $2.5-\varepsilon$ &&--&$3-\varepsilon$ \\
        \hline
	\end{tabular}
\end{table}

\begin{table}[tbhp]
	\centering
	\caption{The temporal discretization   convergence orders for model with nonlinear term ($3$) and initial values (\rmnum{1})-(\rmnum{4}). Theoretical convergence orders are listed in the last row.}
    \label{tab:fun3}
	\begin{tabular}{ccccccccccccc}
		\hline
		\multirow{2}{*}{$N$} & & \multicolumn{2}{c}{initial data (\rmnum{1})} &&  \multicolumn{2}{c}{initial data (\rmnum{2})}  && \multicolumn{2}{c}{initial data (\rmnum{3}) } && \multicolumn{2}{c}{initial data (\rmnum{4})}
		\\ \cline{3-4} \cline{6-7} \cline{9-10} \cline{12-13} 
		~& & Error & Order &&  Error & Order &&  Error & Order &&  Error & Order 
		\\ \hline
		$2^6$ & ~ & 5.129E-06 & -- & ~ & 4.884E-06 & -- & ~ & 1.046E-07 & -- & ~ & 1.029E-07 & -- \\ 
        $2^7$ & ~ & 1.739E-06 & 1.561  & ~ & 1.620E-06 & 1.593  & ~ & 1.873E-08 & 2.482  & ~ & 1.343E-08 & 2.938   \\ 
        $2^8$ & ~ & 5.866E-07 & 1.568  & ~ & 5.416E-07 & 1.580  & ~ & 3.338E-09 & 2.489  & ~ & 1.747E-09 & 2.942   \\ 
        $2^9$ & ~ & 1.953E-07 & 1.587  & ~ & 1.809E-07 & 1.582  & ~ & 5.953E-10 & 2.487  & ~ & 2.275E-10 & 2.941   \\ 
        --&&-- & $1.5-\varepsilon$ &&--& $1.5-\varepsilon$ && --& $2.5-\varepsilon$ &&--&$3-\varepsilon$ \\
        \hline
	\end{tabular}
\end{table}

\section{Conclusions}\label{sec:Conclusions}
In this paper, we have developed a comprehensive and rigorous numerical analysis framework for a class of semilinear parabolic problems subjected to nonsmooth initial data. By employing a linear Galerkin finite element method in space and a high-order explicit exponential Runge-Kutta (EERK) method in time, we systematically investigated the severe order reduction phenomenon inherent in nonsmooth settings. The primary mathematical difficulty of bounding the higher-order Fréchet derivatives of the Nemytskii operator was rigorously resolved through a combination of fractional power space techniques and analytic semigroup estimates. Ultimately, our theoretical analysis establishes a sharp temporal convergence rate of $\min(1 + \gamma/2 + \rho_1(\gamma)/2, \:p)$, which strictly adapts to the exact regularity of the initial data, thereby significantly improving upon suboptimal estimates frequently encountered in the existing literature.

The theoretical framework established herein not only elucidates the intricate interplay between strong nonlinearity and strong stiffness but also provides a robust and extensible foundation for future investigations. Natural continuations of this work include adapting the current analysis to phase-field models governed by the Allen-Cahn and Cahn-Hilliard equations, where handling non-linear stiffness and phase separation under low-regularity conditions remains profoundly challenging. Furthermore, investigating the convergence behavior of fully implicit or exponential Rosenbrock schemes within this fractional Sobolev framework constitutes another important direction for our forthcoming research.








\section*{Funding}
This work was supported by the Guangdong Basic and Applied Basic Research Foundation of China under Grand No.2026A1515012144.

\section*{Author Contributions}
S. Yang performed conceptualization, methodology, software development and wrote the original draft. R. Zhang carried out validation, formal analysis and revised the manuscript. Z. Yu completed formal analysis and validation. J. Fang contributed to conceptualization, methodology and revised the manuscript. All authors reviewed the manuscript.

\section*{Declarations}

\paragraph*{Data Availability}
No datasets were generated or analysed during the study.

\paragraph*{Conflicts of Interest}
The authors declare that they have no conflict of interest.

\paragraph*{Competing interests}
The authors declare no competing interests.

\bibliographystyle{elsarticle-num} 
\bibliography{paper.bib}

\end{document}